\documentclass[12pt]{amsart}

\usepackage{mathrsfs}
\usepackage{amsmath}
\usepackage{amsthm}
\usepackage{amsfonts}
\usepackage{url}
\usepackage[usenames, dvipsnames]{color}
\usepackage{verbatim}
\usepackage{tikz}
\usepackage{tikz-cd}
\usepackage{subcaption}
\usepackage{pgfplots}
\usepackage{adjustbox}
\usepackage{bbm}
\usepackage{pdfsync}
\usepackage{extpfeil}

\usepackage{amssymb}
\usepackage[margin=1in]{geometry}
\usepackage{enumerate}
\usepackage[pdftex,bookmarks=true]{hyperref}
\usepackage{fancyhdr}
\usepackage{graphicx}
\usepackage{xcolor}

\theoremstyle{plain}
\newtheorem{theorem}{Theorem}[section]
\newtheorem{corollary}[theorem]{Corollary}
\newtheorem{lemma}[theorem]{Lemma}

\newtheorem{definition}[theorem]{Definition}
\theoremstyle{definition}
\newtheorem{remark}[theorem]{Remark}

\newtheorem{claim}[theorem]{Claim}

\newcommand{\E}{\mathbb E}
\newcommand{\N}{\mathbb N}
\newcommand{\R}{\mathbb R}
\newcommand{\C}{\mathbb C}
\newcommand{\bP}{\mathbb P}
\newcommand{\eps}{\varepsilon}
\newcommand{\ep}{\varepsilon}
\renewcommand{\P}{\mathbb P}

\newcommand{\Var}{\textup{Var}}
\newcommand{\wps}{\widetilde P_n^*}
\newcommand{\ps}{P_n^*}
\newcommand{\txi}{\tilde \xi}
\def\dis{\displaystyle}

\newcommand\Bk{{\mathbf k}}

\newcommand\Bx{{\mathbf x}}
\newcommand\By{{\mathbf y}}

\newcommand\Bet{\boldsymbol{\eta}}
\DeclareMathOperator{\diag}{diag}

\newcommand\CA{{\mathcal A}}

\newcommand\CE{{\mathcal E}}

\numberwithin{equation}{section}

\begin{document}
	
\title{Real roots of random orthogonal polynomials with exponential weights}
\author{Yen Do, Doron Lubinsky, Hoi H. Nguyen, Oanh Nguyen, and Igor Pritsker}
\address{Department of Mathematics\\ The University of Virginia\\ 141 Cabell Drive, Charlottesville, VA 22904, USA}
\email{yendo@virginia.edu}
 
\address{School of Mathematics, Georgia Institute of Technology, Atlanta, GA 30332, USA}
\email{lubinsky@math.gatech.edu}

\address{Department of Mathematics\\ The Ohio State University \\ 231 W 18th Ave \\ Columbus, OH 43210 USA}
\email{nguyen.1261@osu.edu}

\address{Division of Applied Mathematics\\ Brown University\\  Providence, RI 02906, USA}
\email{oanh\_nguyen1@brown.edu}

\address{Department of Mathematics, Oklahoma State University, Stillwater, OK 74078, USA}
\email{igor@math.okstate.edu}

\begin{abstract}
 We consider random orthonormal polynomials
	$$
	P_{n}(x)=\sum_{i=0}^{n}\xi_{i}p_{i}(x),
	$$
	where $\xi_{0}$, . . . , $\xi_{n}$ are independent random variables with zero mean, unit variance and uniformly bounded $(2+\ep_0)$-moments, and $\{p_n\}_{n=0}^{\infty}$ is the system of orthonormal polynomials with respect to a general exponential weight $W$ on the real line. This class of orthogonal polynomials includes the popular Hermite and Freud polynomials. We establish universality for the leading asymptotics of the expected number of real roots of $P_n$, both globally and locally. In addition, we find an almost sure limit of the measures counting all roots of $P_n.$ This is accomplished by introducing new ideas on applications of the inverse Littlewood-Offord theory in the context of the classical three term recurrence relation for orthogonal polynomials to establish anti-concentration properties, and by adapting the universality methods to the weighted random orthogonal polynomials of the form $W P_n.$ 
\end{abstract}
\maketitle

\section{Introduction and main results}
The study of roots of random polynomials has a long history which can traced back to the fundamental papers by Bloch and Polya in 1932 \cite{BP1}, Littlewood-Offord \cite{LO1,LO2,LO3}, and Kac \cite{Kac1943average}  in the 1940s, Erd\H{o}s--Offord \cite{EO} in 1956, the list goes on. Given a basis of deterministic functions $p_0, p_1, \dots$, consider random polynomials of the form
$$P_n(x)= \xi_0 p_0(x)+ \xi_1 p_1(x) +.... +\xi_n p_n(x)$$
where $\xi_i$ are independent random variables. Three popular classes of random polynomials include
	\begin{itemize}
		\item Random algebraic polynomials, which correspond to $p_n = a_n x^{n}$,
		\item Random trigonometric polynomials, which correspond to $p_n = a_n \sin (nx)$ and $p_n = a_n \cos (nx)$,
		\item Random orthogonal polynomials, which correspond to an orthonormal basis of polynomials $p_n$ associated with a given measure.
	\end{itemize}
A very incomplete list of references for the first two classes include \cite{angstpoly, azais2015local, bally2017non, BP1, DHV, EO, flasche, HKPV, iksanov2016local,  kabluchko2014asymptotic, Kac1943average, LO1,LO2,LO3, nguyenvurandomfunction17, pritsker1, sodin2005zeroes, TVpoly}.
 	In this paper, we focus on random orthogonal polynomials.  Let $\mu$ be a finite positive measure on the {\it real line} with finite moments of all orders, and with the associated orthonormal system of polynomials $\{p_n\}_{n=0}^{\infty}$. In other words, each $p_n$ has degree $n$, real coefficients, and a positive leading coefficient, and furthermore
	$$\int_{\R} p_n(x)p_m(x)d\mu(x) = \begin{cases} 1, & m=n;\\ 0, & m\ne n;\end{cases}$$
	see, e.g., Freud \cite{freudorthogonal} for details.
	Assume furthermore that $\mu$ is absolutely continuous with respect to the Lebesgue measure, that is $d\mu(x)=w(x)dx$, where $w$ is referred to as the weight function. Some popular systems of orthogonal polynomials are
	\begin{enumerate}
		\item [(E1)] Legendre polynomials: The weight $w(x)$ is the indicator function of $[-1,1]$.\\
		\item [(E2)]Chebysev polynomials of type 1:  $w(x)=\frac 1{\sqrt{1-x^2}}$ for $x\in (-1,1)$, and zero elsewhere.\\
		\item [(E3)]Chebysev polynomials of type 2:  $w(x)=\sqrt{1-x^2}$ for $x\in (-1,1)$, and zero elsewhere.\\
		\item [(E4)]Jacobi polynomials: These are generalizations of Chebyshev and Legendre polynomials, with $w(x)=(1-x)^\alpha (1+x)^\beta$ for $x\in (-1,1)$ and zero elsewhere, where $\alpha,\beta>-1$.\\
		\item [(E5)]Hermite polynomials:  $w(x)=e^{-x^2}$ for $x\in \R$.\\
		\item [(E6)] Freud polynomials: These are generalizations of  Hermite polynomials, with $w(x)=e^{-c|x|^{\lambda}}$ for $x\in \R$, where $c>0,\lambda>0$.\\
	\end{enumerate}

One of the most basic questions in the study of roots is the average number of real roots of $P_n$. In \cite{Kac1943average}, Kac derived the celebrated Kac-Rice formula that allows us to reduce this task to evaluating an integral, which is often most manageable when the random variables $\xi_i$ are standard Gaussian.
Let $N_n(S)$ be the number of roots of $P_n$ in a set $S$. In this Gaussian setting, the average number of real roots of $P_n$ has been known for the following models.
	\begin{itemize}
		\item Random Legendre polynomials: Das \cite{das1971real} found that $\E[N_n(-1,1)]$ is asymptotically equal to $n/\sqrt{3}$. Wilkins \cite{wilkins1997expected} improved the error term in this asymptotic relation by showing that $\E[N_n(-1,1)] = n/\sqrt{3} + o(n^\ep)$ for any $\ep>0$;
		\item Random Jacobi polynomials: Das and Bhatt \cite{das1982real} concluded that $\E[N_n(-1,1)]$ is also asymptotically equal to $n/\sqrt{3}$;
		\item More generally, when $\mu$ a finite Borel measure with {\it compact support} on the real line. Lubinsky--Pritsker--Xie  \cite{LPX1} showed that under mild conditions on the weight $d\mu/dx$, the same asymptotic holds.
		\item Beyond compactly supported measures, Pritsker--Xie \cite{PX} extended this asymptotic to the random Freud polynomials. Later, Lubinsky--Pritsker--Xie \cite{LPX2} showed that the same asymptotic holds for general exponential weights (see Theorem \ref{thm:gau:global} below).
	\end{itemize}

	 Little is known about the average number of roots when the coefficients are not Gaussian. In \cite{DOVorthonormal}, Do--Nguyen--Vu extended the asymptotics for the model in \cite{LPX1} to general random variables with mean 0, variance 1, and bounded $(2+\ep_0)$-moments for some $\ep_0>0$. In other words, under some conditions, the asymptotic of $\E N_n(\R)$ is known when the measures $\mu$ are compactly supported, such as the measures in Examples (E1)--(E4) above.
	
	 This paper is devoted to random orthogonal polynomials associated with exponential weights on the whole real line, such as random polynomials spanned by Hermite and Freud polynomials mentioned in (E5) and (E6). In particular, we show that the results of \cite{LPX2} hold for general random coefficients $\xi_i$. We shall assume throughout the paper that the measure $\mu$ is absolutely continuous, with $d\mu(x)=W^{2}(x)dx$, where the weight $W$ belongs to the class $\mathcal{F}(C^2)$ of exponential weights considered in \cite{Lubinskybookexp} and \cite{LPX2}.

\begin{definition} [Weight class $\mathcal{F}(C^2)$]\label{def1.1}
Let $W=e^{-Q}$, where $Q: \R \rightarrow [0, \infty)$ satisfies the following conditions:\\
(a) $Q'$ is continuous in $\R$ and $Q(0)=0$.\\
(b) $Q'$ is non-decreasing in $\R$, and $Q''$ exists in $\R\setminus\{0\}$.\\
(c) \[\displaystyle \lim_{|t|\rightarrow \infty} Q(t)=\infty.\]
(d) The function
\[
T(t)=\frac{tQ'(t)}{Q(t)},\ t \neq 0,
\]
is quasi-increasing in $(0, \infty)$, in the sense that for some $C_1>0$,
\[
0<x<y \Rightarrow T(x) \leq C_1 T(y).
\]
We assume an analogous restriction for $y<x<0$. In addition, we assume that for some $\Lambda >1$,
\[
T(t)\geq \Lambda \text{ in } \R\setminus\{0\}.
\]
(e) There exists $C_2>0$ such that
\[
\displaystyle \frac{Q''(x)}{|Q'(x)|}\leq C_2\frac{|Q'(x)|}{Q(x)}, \ x\in \R\setminus\{0\}.
\]
Then we write $W \in \mathcal{F} (C^2)$.
\end{definition}

It is straightforward to check that the Freud weights $w(x)=e^{-c|x|^{\lambda}}$ of (E6) belong to $\mathcal{F} (C^2)$ when $\lambda>1.$
From the definition, $\mathcal{F} (C^2)$ consists of smooth functions $e^{-Q}$, where $Q$ is convex, and has regular growth behavior at infinity.
We note that $\mathcal{F} (C^2)$ is one of the largest known classes of weights that allow for a very complete theory of the associated orthogonal polynomials,
see \cite{Lubinskybookexp} for a comprehensive exposition.

We recall the result of \cite{LPX2} on the asymptotic for $\E N_n(\R)$ where $N_n(\R)$ is the number of real roots of the random polynomial
$$
P_{n}(x)=\sum_{i=0}^{n}\xi_{i}p_{i}(x).
$$

\begin{theorem} [Gaussian setting, cf. Theorem 2.2 of \cite{LPX2}] \label{thm:gau:global}
Let $W=e^{-Q}\in \mathcal{F} (C^2)$, where $Q$ is even. If the function $T$ in the definition of $\mathcal{F} (C^2)$ satisfies
\begin{equation}\label{1.3}
\lim_{x\rightarrow \infty}T(x) = \alpha,
\end{equation}
where $\alpha\in(1, \infty]$, then the expected number of real zeros of $P_n$, with $\{\xi_i\}_{i=0}^\infty$ being i.i.d. standard Gaussian, satisfies
\begin{align} \label{1.4}
\lim_{n\to\infty} \frac{1}{n} \E[N_n(\R)]= \frac{1}{\sqrt{3}}.
\end{align}
\end{theorem}

To state the result on the local asymptotics for the expected number of real zeros, we first define the Ullman distribution that provides the limiting density for the number of real roots.
Following \cite{Lubinskybookexp} and \cite{ST}, we define the Ullman distribution $\mu_{\alpha},\ \alpha>1,$ by
\begin{equation}\label{Ullman}
d \mu_{\alpha}(x)= \left(\frac{\alpha}{\pi}\int_{|x|}^1\frac{t^{\alpha-1}}{\sqrt{t^2-x^2}}\,dt \right) dx,\quad x\in[-1,1],
\end{equation}
and we consider the following contracted version of $P_n$
\begin{align} \label{1.5}
P_n^*(s):=P_n(a_ns),\quad n\in\N,
\end{align}
where $a_n$ is the Mhaskar-Rakhmanov-Saff number associated with the weight $W$, see \cite{Lubinskybookexp} and \cite{ST}.

 Roughly speaking, most of the real roots of $P_n^*$ stay in $[-1, 1]$.  For any set $E\subset\C$, $N_n^*(E)$ denotes the number of zeros of a random polynomial $P_n^*$ located in $E$. We now state the local result on the asymptotic of $\E[N_n^*\left([a,b]\right)]$ for intervals $[a,b]\subset (-1, 1)$. It says that $\mu_\alpha$ is the limiting measure, up to a constant.

\begin{theorem} [Gaussian setting, see Theorem 2.3 in \cite{LPX2}]\label{thm:gau:local}
Assume that $\{\xi_k\}_{k=0}^n$ are i.i.d. standard Gaussian. Let $W=e^{-Q} \in \mathcal{F} (C^2)$, where $Q$ is even. Suppose that the function $T$ in the definition of $\mathcal{F} (C^2)$ satisfies \eqref{1.3}. For any $[a,b] \subset (-1,1)$, we have
\begin{equation} \label{1.6}
\lim_{n\rightarrow\infty} \frac{1}{n}\E\left[N_n^*\left([a,b]\right)\right] = \frac{1}{\sqrt{3}} \mu_{\alpha}([a,b]).
\end{equation}
\end{theorem}

It is worth noting that condition \eqref{1.3} plays a crucial role in the existence of the limit in \eqref{1.6}. In fact, if the weight function $W$ does not behave at infinity in a regular way,
by for example oscillating between the values of $e^{-c|x|^{\lambda_1}}$ and $e^{-c|x|^{\lambda_2}}$ for $x$ near infinity, with $\lambda_1>\lambda_2>1,$ then the limit in \eqref{1.6} may fail to exist. This indicates that the local asymptotic behavior of $\E\left[N_n^*\left([a,b]\right)\right]$ may not abide a single limiting distribution, and may differ dramatically when the interval $[a,b]$ is shifted.

We are now ready to state our main universality result confirming that Theorems \ref{thm:gau:global} and \ref{thm:gau:local} hold for general random coefficients.  
\begin{theorem} \label{thm:main}
Assume that the coeffcients $\{\xi_k\}_{k=0}^n$ are independent real-valued random variables such that $\E[\xi_k]=0, \Var[\xi_k]=1,$ and $\E[|\xi_k|^{2+\ep_0}] < C$ for some constants $C,\ep_0>0$ and all $k=0,1,\dots$. Let $W=e^{-Q}\in \mathcal{F} (C^2)$, where $Q$ is even. If the function $T$ in the definition of $\mathcal{F} (C^2)$ satisfies \eqref{1.3} with $\alpha\in(1,\infty)$, then we have
\begin{enumerate}
	\item [(a)] The expected number of real zeros of random orthogonal polynomial $P_n$ satisfies \eqref{1.4}.
	\item [(b)] For any $[a,b] \subset (-1,1)$, the real zeros of polynomials $P_n^*$ satisfy \eqref{1.6}.
\end{enumerate}
\end{theorem}
 
In fact we will show that the $k$-correlations of the real roots of $P_n^\ast$ are universal for any fixed $k$, see Theorem \ref{thm:general:real}.

As a direct application, we obtain the statement for random Freud polynomials.
\begin{corollary}
Assume that the coeffcients $\{\xi_k\}_{k=0}^n$ are independent real-valued random variables such that $\E[\xi_k]=0, \Var[\xi_k]=1,$ and $\E[|\xi_k|^{2+\ep_0}] < C$ for some constants $C,\ep_0>0$ and all $k=0,1,\dots$. For any constants $\lambda>1$ and $c>0$, consider the random orthogonal polynomial $P_n$ associated with the Freud weight $w(x) = e^{-c|x|^{\lambda}}$ for $x\in \R$.  Then the expected number of real zeros of  $P_n$ satisfies \eqref{1.4}. Moreover, if $[a,b] \subset (-1,1)$ then the real zeros of polynomials $P_n^*$, defined in \eqref{1.5}, satisfy \eqref{1.6} with $\alpha = \lambda$.
\end{corollary}

Next, define the normalized zero counting measure $\tau_n=\frac{1}{n}\sum_{k=1}^n \delta_{z_{k,n}}$ for the scaled  polynomial $P_n^*$ \eqref{1.5}, where $\{z_{k,n}\}_{k=1}^n$
are its zeros in $\C$, counted with multiplicities, and $\delta_z$ denotes the unit point mass at $z$. We find the weak* limit of $\tau_n$ in the case of random coefficients satisfying the assumptions of Theorem
\ref{thm:main}, and thus extend Theorem 2.4 of \cite{LPX2}.

\begin{theorem} \label{thm:zeroasympt}
Assume that the coefficients $\{\xi_k\}_{k=0}^n$ are independent real-valued random variables such that $\E[\xi_k]=0, \Var[\xi_k]=1,$ and $\E[|\xi_k|^{2+\ep_0}] < C$ for some constants $C,\ep_0>0$ and all $k=0,1,\dots$. Let $W=e^{-Q}\in \mathcal{F} (C^2)$, where $Q$ is even. If the function $T$ in the definition of $\mathcal{F} (C^2)$ satisfies \eqref{1.3} with $\alpha\in(1,\infty)$, then the normalized zero counting measures $\tau_n$ for the scaled  polynomials $P_n^*(s)$ converge to $\mu_{\alpha}$ of \eqref{Ullman} in the weak* topology with probability one.
\end{theorem}

The following corollary is useful in the proof of Theorem \ref{thm:main}.

\begin{corollary} \label{cor:explimit}
Suppose that the assumptions of Theorem \ref{thm:zeroasympt} hold. If $E\subset\C$ is any compact set satisfying $\mu_\alpha(\partial_{\C} E)=0,$ then
\begin{equation} \label{explimit}
\lim_{n\rightarrow\infty} \frac{1}{n}\E\left[N_n^*(E)\right] = \mu_\alpha(E).
\end{equation}
\end{corollary}

\begin{remark}
In Theorem \ref{thm:zeroasympt} and Corollary \ref{cor:explimit}, the conditions on the random coefficients can be relaxed. In the proof for these results, we only need to assume that the random variables $\{\xi_k\}_{k=0}^n$ are independent with mean 0, variance 1, and there exist constants $A>0, a\in (0, 1)$ such that
\[
\E[\log^+|\xi_k|] < A, \quad \P(|\xi_k-x|\le a)\le 1-a,\quad k=0,1,\dots.
\]
\end{remark}

\section{The method of proof and universality of the correlations}

To establish Theorem \ref{thm:main}, we apply the universality framework of \cite{nguyenvurandomfunction17}. This approach was earlier developed in Tao--Vu \cite{TVpoly}, Do--Nguyen--Vu \cite{DOV}, and it now becomes one of the standard methods in the field. The global idea is to compare the distribution of roots of $\ps$ with that of the Gaussian counterpart, namely
$$\wps(x) = \sum_{k=0}^{n} \txi_k p_k(a_n x)$$
where $\txi_k$ are iid standard Gaussian. Generally speaking, if none of the $p_k$ dominates the sum (see the Delocalization Condition C3 below), then the effect of changing the distribution of the coefficients is small and so the roots of $\ps$ have roughly the same distribution as those of $\wps$. For technical reasons, and in particular to tame the growth of the orthogonal polynomials $p_k$, we shall work with the random functions
\[
F_n(x) = P_n^*(x) W(a_n x) = \sum_{k=0}^{n} \xi_k p_k(a_n x) W(a_n x) = \sum_{k=0}^{n} \xi_k f_k(x),
\]
where $f_k(x) = p_k(a_n x) W(a_n x),\ k=1,\ldots,n.$ Since $W(x)\neq 0,\ x\in\R,$ the real zeros of $F_n$ and $P_n^* $ are clearly identical. Further, we need an extension of the weight $W$ into the complex plane, which is convenient to obtain via the potential theoretic argument below. Following \cite{Lubinskybookexp} and \cite{ST}, we introduce the Mhaskar-Rakhmanov-Saff interval $\Delta_n :=[-a_n, a_n]$. The equilibrium density (or the density of the weighted equilibrium measure) is defined as
\[
\sigma_n(x)=\displaystyle \frac{\sqrt{(x-a_{-n})(a_n-x)}}{\pi^2}\int_{a_{-n}}^{a_n} \frac{Q'(s)-Q'(x)}{s-x}\frac{ds}{\sqrt{(s-a_{-n})(a_n-s)}}, \, x\in \Delta_n.
\]
It satisfies the equilibrium equation \cite[p. 41]{Lubinskybookexp}:
\begin{align} \label{wequil}
\int_{a_{-n}}^{a_n} \log \frac{1}{|x-t|}\sigma_n(t)\, dt + Q(x) = M, \quad x\in \Delta_n.
\end{align}
Note that the weighted equilibrium measure $\sigma_n(x)\, dx$ is supported on $\Delta_n$, and has total mass $n$:
\[
\int_{a_{-n}}^{a_n} \sigma_n(x)\, dx=n.
\]
For details on $\sigma_n$, we refer to the book \cite{Lubinskybookexp}. We now set
\begin{align} \label{extension}
Q(z) := M - U^{\mu_n}(z)\ \mbox{and}\ W(z) := \exp\left(-Q(z)\right), \quad z\in \C\setminus[-a_n,a_n],
\end{align}
where $U^{\mu_n}(x)$ is the logarithmic potential of $\mu_n:$
\[
U^{\mu_n}(x) := - \int_{a_{-n}}^{a_n} \log |x-t|\sigma_n(t)\, dt.
\]
This gives the desired extension of $W$ into $\C$,
with $Q$ being harmonic in $\C\setminus[-a_n,a_n]$ and continuous in $\C,$ cf. \cite{Lubinskybookexp} and \cite{ST}.

The exact conditions for the applicability of the universality method are stated  below. For any constant $\ep>0$, we shall show that there exists a constant $b>0$ such that for any positive constants $c_1,  A$, there exists  $C_1>0$ for which the following conditions hold for the interval $D(\ep) =  [-1+\ep, 1-\ep]$.

\begin{enumerate}
	
    \item [(C1)] {\it Boundedness:} \label{cond-bddn} For any $z\in D(\ep)$, with probability at least $1 - C_1 n ^{-A}$, $|F_n(w)|\le \exp(n^{c_1})$ for all $w\in B_{\C}  (z, 1/n)$.
	
	\item [(C2)] {\it Delocalization:}\label{cond-delocal} For every $z \in D(\ep)+B_{\C} (0, 1/n)$, for every $k = 0, \dots, n$,
	$$\frac{|p_k(a_n z)|}{\sqrt{\sum _{j = 0}^{n}|p_j(a_n z)|^{2}}}\le C_1 n^{-b}.$$
	
	\item [(C3)] {\it Derivative growth:}\label{cond-repulsion} For any $s\in D(\ep)$,
	\begin{equation}
	\sum_{j=0}^{n} |p_j'(a_n s)|^{2} a_n^2 \le C_1 n ^{2+c_1}\sum_{j=0}^{n} |p_j(a_n s)|^{2},\nonumber
	\end{equation}
	\begin{equation}
	\sup_{z\in B_{\C} (s, 1/n)} \sum_{j=0}^{n} |p_j''(a_n z)|^{2} a_n^4 \le C_1 n ^{4+c_1}\sum_{j=0}^{n} |p_j(a_n s)|^{2}.\nonumber
	\end{equation}
	
		\item  [(C4)] {\it Anti-concentration:}\label{cond-smallball1} For every $z\in D(\ep)$, with probability at least $1 - C_1n^{-A}$, there exists $w\in B_{\C} (z, 1/n)$ for which $|F_n(w)|\ge \exp(-n^{c_1})$.

\end{enumerate}

The following key lemma shows that these conditions are satisfied by our $F_n$.
\begin{lemma}\label{lm:conditions}
 	Assume the hypothesis of Theorem \ref{thm:main}. For any constant $\ep>0$, there exists a constant $b>0$ such that for any positive constants $c_1,  A$, there exists  $C_1>0$ for which Conditions (C1)--(C4) hold for the interval $D(\ep) = [-1+\ep, 1-\ep]$.
\end{lemma}

Assuming Lemma \ref{lm:conditions}, we have the following universal result of the correlations of real roots.

\begin{theorem}\label{thm:general:real}
	Assume the hypothesis of Theorem \ref{thm:main}. Let $k$ be any positive integer. Then there exist positive constants $C$ and $c$ independent of $n$ such that the following holds.
	For any real numbers $x_1,\dots, x_k$, all of which are in $D(\ep)$, and any function $G: \mathbb{R}^{k}\to \R$ supported on $\prod_{i=1}^{k}[x_i-c/n, x_i+c/n] $ with continuous derivatives up to order $2k+4$ and $||\triangledown^m G||_{\infty}\le n^{a}$ for all $0\le m\le 2k+4$, we have
	\begin{eqnarray}
	\left |\E\bigg[\sum G\left (\zeta_{i_1}, \dots, \zeta_{i_k}\right)\bigg]
	-\E\bigg[\sum G\left (\tilde \zeta_{i_1}, \dots, \tilde \zeta_{i_k}\right)\bigg] \right |\le C n^{-c},\nonumber
	\end{eqnarray}
	where the first sum runs over all $k$-tuples $(\zeta_{i_1}, \dots, \zeta_{i_k}) \in \R^{k} $ of the roots $\zeta_1, \zeta_2, \dots$ of $P^*_n$, and the second  sum runs over all $k$-tuples $(\tilde \zeta_{i_1}, \dots, \tilde \zeta_{i_k}) \in \R^{k}$ of the roots $\tilde \zeta_1, \tilde  \zeta_2, \dots$ of $ \wps$.
\end{theorem}
 This theorem is a generalization of \cite[Theorem 2.6]{nguyenvurandomfunction17}. We provide a proof in Section \ref{sec:proof:uni}.

It is worth mentioning, by triangle inequality, that Theorem \ref{thm:general:real} also holds true if $\xi_k$ and $\tilde \xi_k$ are any random variables with mean 0, variance 1, and uniformly bounded $(2+\ep_0)$ moments. In other words, it is not necessary that $\tilde \xi_k$ are standard Gaussian. In Section \ref{sec:correlation} we will provide various formulas for the $k$-correlation function of the real roots of $P^*_n$ (or of $P_n$) under the assumption that the random coefficients $\xi_i$ have smooth density functions.

Using Theorem \ref{thm:general:real}, we derive an estimate for the number of real roots at a local scale, which immediately gives the second part of Theorem \ref{thm:main} by decomposing a big interval into intervals of length $O(1/n)$. It is convenient here to use the notation $N_{\ps}(E)$ and $N_{\wps}(E)$ for the number of roots of $\ps$ and $\wps$ respectively in a set $E$.

\begin{theorem}\label{thm:local:expectation}
	Under the hypothesis of Theorem \ref{thm:main}, for any $\ep>0$, there exists a positive constant $c$ such that the following holds for any interval $[a, b]$ of length $O(1/n)$ inside $D(\ep)$:
	\begin{equation*}\label{key}
	\E[N_{\ps}([a, b])] - \E[N_{\wps}([a, b])] \ll n^{-c}.
	\end{equation*}
\end{theorem}

Outside of the set $D(\ep)$, we show that the number of roots is negligible.
\begin{theorem}\label{thm:edge}
	Under the hypotheses of Theorem \ref{thm:main}, we have
	\begin{equation*}\label{key}
	\lim_{\ep\to 0} \lim _{n\to \infty}\frac{\E[N_{\ps}(\R \setminus D(\ep))]}{n} = \lim_{\ep\to 0} \lim _{n\to \infty}\frac{\E[N_{\wps}(\R \setminus D(\ep))]}{n} = 0.
	\end{equation*}
\end{theorem}

 Combining the last two theorems, we obtain that
 \begin{equation*}\label{key}
\lim _{n\to \infty} \left(\frac{\E[N_{\ps}(\R)]}{n} -\frac{\E[N_{\wps}(\R)]}{n}\right) = 0.
 \end{equation*}
 This proves the first part of Theorem \ref{thm:main}, in view of Theorem \ref{thm:gau:global}.

\section{Proof ideas and our contributions}

As mentioned in the previous section, the global structure of our paper is to follow the universality framework that has been developed in \cite{ nguyenvurandomfunction17, TVpoly}, and more recently in \cite{DOVorthonormal} for random orthogonal polynomials with compactly supported measures with a ``nice" density function. However, handling exponential weights is a much more complicated task that requires new tools and ideas.

Firstly, the functions $F_n$ whose universality property is of interest in previous papers were assumed to be analytic, which was important for several of their arguments such as the one using Hal\'asz's inequality. Unfortunately, our function $F_n(x) = P_n(x) W(x)$ is not analytic. We propose a novel way to circumvent this condition in Section \ref{sec:proof:uni}.

Secondly, the framework of universality requires to investigate the behavior of $F_n$ in an open neighborhood of the real line in the complex plane. That means we need to extend and study the behavior of the weight function $W$ on $\C$. This brings technical challenges that one did not have to deal with in the previous papers such as \cite{LPX2}.

Last but not least, one of the most challenging obstacles that we have to face is to show that the function $F_n(x)$ cannot be too small on any interval $I$ of length $1/n$ inside the core interval $(-a_n, a_n)$; this property is known as the anti-concentration property. If $F_n(x)$ were too small, the function would become sensitive to noise, and one would not expect universality. So, it is crucial to establish the anti-concentration. In particular, we want to show that with high probability, there exists $x\in I$ such that
$$|F_n(x)|\ge \exp(-n^{\ep})$$
for some small $\ep>0$. 

When $I$ is very near the edge of $(-a_n, a_n)$ (that is $|x| = \Theta(a_n)$), $W(x)$ can be as small as $\exp(-\Theta(n))$ and hence the above inequality becomes
\begin{equation}\label{eq:P:anti}
|P_n(x)|\ge \exp(\Theta(n^{1})).
\end{equation}
 
This is in sharp contrast to the case when the measure $\mu$ is compact in which one only needs to show
$$|P_n(x)|\ge \exp(-n^{\ep}).$$
The difficulty here is that with the power $\ep$ for $n$, one can restrict to lower degree polynomials $p_k$ for $k\ll n^{2\ep}$ and use classical tools like Remez inequality. This was one of the key ingredients in both \cite{nguyenvurandomfunction17} and \cite{DOVorthonormal}. Now, having the power 1 in \eqref{eq:P:anti}, one can no longer disregard the high degree polynomials and use Remez inequality as before. It boils down to controlling $p_k(x)$ for $k\approx n$ and $|x|\approx a_n$. This is known to have erratic behavior and it may be possible that $p_k$ is very tiny on the entire $I$ which is a significant obstruction.

In this paper, we come up with a completely new argument that incorporates the inverse Littlewood-Offord theory, the polynomial recurrence structure of $p_k$, and Remez inequality. The general idea is that if $P_n$ has a small anti-concentration probability, then $p_k$ must exhibit an algebraic structure. We then show that it is impossible for $p_k$ to have both algebraic structure and polynomial recurrence structure. We refer the details to Section \ref{sec:cond1}, where we also include a sketch of proof. We note that certain recurrence was also used in \cite[Section 13]{TVpoly} but our implementation here is completely orthogonal and different. 

Finally, we obtain the $k$-correlation function formulas (in Section \ref{sec:correlation}) by adapting the method of \cite{GKZ} with some minimal modifications toward orthogonal polynomials.

 \textbf{Organization.} To prove Lemma \ref{lm:conditions}, in Sections \ref{sec:cond2}--\ref{sec:cond1}, we verify Conditions (C1)--(C4). In Section \ref{sec:proof:local}, we prove Theorem \ref{thm:local:expectation}. Section \ref{sec:zerodistr} contains a proof of Theorem \ref{thm:zeroasympt}. Theorem \ref{thm:edge} is proven in Section \ref{sec:proof:edge}. We conclude with explicit formulas for $k$-correlation function in Section \ref{sec:correlation}.

\textbf{Notations.} We use standard asymptotic notations under the assumption that $n$ tends to infinity.  For two positive  sequences $(a_n)$ and $(b_n)$, we say that $a_n \gg b_n$ or $b_n \ll a_n$ if there exists a constant $C$ such that $b_n\le C a_n$. If $|c_n|\ll a_n$ for some sequence $(c_n)$, we also write $c_n\ll a_n$.

If $a_n\ll b_n\ll a_n$, we say that $b_n=\Theta(a_n)$.  If $\lim_{n\to \infty} \frac{a_n}{b_n} = 0$, we say that $a_n = o(b_n)$.  If $b_n\ll a_n$, we sometimes employ the notations $b_n = O(a_n)$ and $a_n = \Omega(b_n)$ to make the idea intuitively clearer or the writing less cumbersome; for example, if $A$ is the quantity of interest, we may write $A = A'+ O(B)$ instead of $A - A' \ll B$, and $A = e^{O(B)}$ instead of $\log A\ll B$. We also write that $a_n=O_C(b_n)$ if the implied constant depends on a given parameter $C$.

For the orthonormal polynomials $\{p_j(x)\}_{j=0}^{\infty}$, we define the {\bf reproducing kernel} by
\[
K_n(x,y)=\sum_{j=0}^{n}p_j(x)p_j(y),
\]
and the differentiated kernels by
\[
K_n^{(k,l)}(x,y)=\sum_{j=0}^{n}p_j^{(k)}(x)p_j^{(l)}(y),\quad k,l\in\N\cup\{0\}.
\]
It is apparent that the above Conditions (C1)--(C4) can be rephrased in term of reproducing kernels.

\section{Partial proof of Lemma \ref{lm:conditions}: Boundedness}\label{sec:cond2}
In this section, we prove Condition (C1) as the first step in proving Lemma \ref{lm:conditions}.
Let $$J_n(\ep) = a_n D(\ep) = [-(1-\ep)a_n, (1-\ep)a_n].$$
Since $P_n$ is spanned by the orthonormal basis $\{p_k\}_{k=0}^n$, we have that
\[
\|W P_n\|_{L^2(\R)} = \left(\int_{-\infty}^{\infty} |P_n(x)|^2 W^2(x)\,dx\right)^{1/2}= \left(\sum_{k=0}^n |\xi_k|^2\right)^{1/2}.
\]
Applying the Nikolskii-type inequality (cf. \cite[p. 60]{LPX2}), valid for any polynomial $Q_n$ with complex coefficients of degree at most $n$,
\[
\left\|W Q_n\right\|_{L^{\infty}(\R)}\leq C n\left\|W Q_n \right\|_{L^2(\R)},\quad n\in\N,
\]
we obtain that
\begin{align} \label{3.0}
\left\|W P_n\right\|_{L^{\infty}(\R)} \leq C n \left(\sum_{k=0}^n |\xi_k|^2\right)^{1/2},\quad n\in\N.
\end{align}
It follows from Lemma \ref{lem3.1} below (with $m=0$) that there is a constant $B>0$ such that for all
$n\in\N,\ s\in D(\ep)=[-1+\ep, 1-\ep],$ we have
\begin{align} \label{3.1}
W(a_n s) \sup_{z\in B_{\C} (s, 1/n)}|P_n(a_n z)| \le B  \left\|W P_n\right\|_{L^{\infty}(\R)} \le C n \left(\sum_{k=0}^n |\xi_k|^2\right)^{1/2}.
\end{align}
From Chebyshev's inequality for random coefficients with finite $(2+\ep_0)$-moments,
we obtain that $\bP(|\xi_k|>n^{(A+1)/2}) = O\left(n^{-(A+1)}\right)$ holds for any $A>0,\ k=0,1,\ldots,n.$ Hence
\[
\left(\sum_{k=0}^n |\xi_k|^2\right)^{1/2} \le n^{(A+1)/2}\sqrt{n+1}
\]
and
\[
W(a_n s) \sup_{z\in B_{\C} (s, 1/n)}|P_n(a_n z)| \le C n^{2+A/2},\quad n\in\N,\ s\in D(\ep),
\]
with probability at least $(1-cn^{-(A+1)})^{n+1}\ge 1-d/n^A,\ d>0$. Since
\[
\sup_{z\in B_{\C} (s, 1/n)} |W(a_n z)| \le C W(a_n s)
\]
by our choice of the extension for $W$ to the complex plane, we have that
\[
\sup_{z\in B_{\C} (s, 1/n)}|P_n(a_n z)W(a_n z)| \le \sup_{z\in B_{\C} (s, 1/n)}|P_n(a_n z)|  \sup_{z\in B_{\C} (s, 1/n)} |W(a_n z)| \le C n^{2+A/2}
\]
for all $n\in\N,\ s\in D(\ep),$ with probability at least $1-d/n^A,\ d>0$. This verifies the boundedness condition (C1) 
because $W(a_n z) P_n(a_n z) = F_n(z)$ and $n^{2+A/2} = o\left(\exp(c_1 n)\right)$ for all $c_1>0.$

We now state and prove the following lemma about growth of weighted polynomials and their derivatives in the complex plane. It was already used
in the above proof, namely in \eqref{3.1}, and will also be very useful in the other sections.

\begin{lemma} \label{lem3.1}
	For all $n\in\N,\ s\in D_n(\ep),$ and polynomials $R_n,\ \deg(R_n)\leq n,$ there is a constant $B>0$ such that
	\[
	W(a_n s) \sup_{z\in B_{\C} (s, 1/n)}|R_n^{(m)}(a_n z)| \le B \left(\frac{n}{a_n}\right)^m \sup_{t\in[-1,1]}|W(a_n t) R_n(a_n t)|, \quad m=0,1,2,\ldots.
	\]
\end{lemma}

\begin{proof}[Proof of Lemma \ref{lem3.1}]
	We begin with a general estimate for the growth of an arbitrary polynomial $R_n,\ \deg(R_n)\leq n,$ in the complex plane in terms of its weighted supremum norm on $\Delta_n=[-a_n,a_n].$ Recall that the potential $U^{\mu_n}(x) := - \int_{-a_n}^{a_n} \log |x-t|\sigma_n(t)\, dt$ of the weighted equilibrium measure $\mu_n(x):=\sigma_n(x)\,dx$ satisfies the equilibrium equation from \eqref{wequil}
	\[
	U^{\mu_n}(x) + Q(x)= M, \quad x\in \Delta_n.
	\]
		It is known that $U^{\mu_n}(x)$ is superharmonic and continuous in $\C$, harmonic in $\C\setminus\Delta_n$, and behaves like $-n\log|x|$ near infinity. Therefore the function
	\[
	h(z) := \log|R_n(z)| - \log\|R_n W\|_{L^\infty(\Delta_n)} + U^{\mu_n}(z) - M
	\]
	is subharmonic in the domain $\overline{\C}\setminus\Delta_n$, and has the following boundary values
	\begin{align*}
	h(x) &= \log|R_n(x)| - \log\|R_n W\|_{L^\infty(\Delta_n)} - Q(x) \\
	&= \log|R_n(x) W(x)| - \log\|R_n W\|_{L^\infty(\Delta_n)} \le 0,\quad x\in\Delta_n.
	\end{align*}
	Appying the Maximum Principle for subharmonic functions, we conclude that $h(z) \le 0$ for all $z\in\overline{\C}\setminus\Delta_n$, which is equivalent to
	\begin{align}\label{5.4}
	|R_n(z)| \le \|R_n W\|_{L^\infty(\Delta_n)} \exp\left(M - U^{\mu_n}(z)\right),\quad z\in\C.
	\end{align}
	This upper bound allows us to estimate the growth of $R_n$, provided we have appropriate estimates on the potential $U^{\mu_n}(z).$ Required estimates follow from Lemma 5.10(a) of \cite[p. 130]{Lubinskybookexp}: There is a constant $C>0$ such that
	\begin{align}\label{5.5}
	U^{\mu_n}(x+iy) - U^{\mu_n}(x) = O(1)
	\end{align}
	holds uniformly for $n\ge n_0,\ x\in J_n(\ep),$ and $|y| \le C a_n/n.$ Furthermore, for any pair $x_1,x_2\in\Delta_n$ such that $|x_1 - x_2|\le c a_n/n$, we obtain that
\begin{equation}\label{eq:Q}
	|U^{\mu_n}(x_1) - U^{\mu_n}(x_2)| = |Q(x_1) - Q(x_2)| \le |x_1 - x_2|\,\max_{x\in\Delta_n} |Q'(x)| \le c \frac{a_n Q'(a_n)}{n} = O(1),
\end{equation}
	where we used that $a_n Q'(a_n) = O(n)$ by (3.7) and (3.11) of \cite{LPX2}. The latter estimate together with \eqref{5.5} implies that
	\[
	U^{\mu_n}(z) = U^{\mu_n}(x) + O(1) = -Q(x) + O(1),\quad z\in B_{\C} (x, c a_n/n),\ x\in J_n(\ep),
	\]
	where $c>0$ is independent of $n,x,z$. Hence \eqref{5.4} now gives
	\begin{align}\label{5.6}
	|R_n(z)| &\le C \|R_n W\|_{L^\infty(\Delta_n)}  \exp\left(M - U^{\mu_n}(x)\right) \\ &= C \frac{\|R_n W\|_{L^\infty(\Delta_n)}}{W(x)},\quad z\in B_{\C} (x, c a_n/n),\ x\in J_n(\ep). \nonumber
	\end{align}
	Differentiating Cauchy's integral formula, we deduce from \eqref{5.6} that
	\begin{align}\label{5.7}
	|R_n^{(m)}(z)| &= \left| \frac{m!}{2\pi i} \int_{|w-x|=2 a_n/n} \frac{R_n(w)\,dw}{(w-z)^{m+1}} \right| \\ \nonumber
	&\le C \left(\frac{n}{a_n}\right)^m \frac{\|R_n W\|_{L^\infty(\Delta_n)}}{W(x)},\quad z\in B_{\C} (x, a_n/n),\ x\in J_n(\ep).
	\end{align}
	The statement of Lemma \ref{lem3.1} now follows by changing variable $x=a_n s$, while passing from $\Delta_n$ to $[-1,1]$ and from $J_n(\ep)$ to $D_n(\ep)$.
\end{proof}

\section{Partial proof of Lemma \ref{lm:conditions}: Delocalization}\label{sec:cond3}
In this section, we prove Condition (C2) as the next step in proving Lemma \ref{lm:conditions}.

Recall that $K_n(x,x)=\sum_{j=0}^{n} p_j^{2}(x),$ so that $|K_n(x,x)| = K_n(x,x),\ x\in\R,$
and $|K_n(w,w)| \le \sum_{j=0}^{n} |p_j(w)|^{2},\ w\in\C.$ We also use the notation
$J_n(\ep) = a_n D(\ep) = [-(1-\ep)a_n, (1-\ep)a_n].$
In order to prove the delocalization condition, we write
\begin{align} \label{4.1}
\frac{|p_k(w)|}{\sqrt{\sum_{j=0}^{n} |p_j(w)|^{2}}} \le \frac{|p_k(w)|}{\sqrt{|K_n(w,w)|}} = \frac{|p_k(w)| W(x)}{\sqrt{K_n(x,x) W^2(x)}}
\left(\frac{K_n(x,x)}{|K_n(w,w)|}\right)^{1/2},
\end{align}
where we assume that $x\in J_n(\ep)$ and $z\in\C$. Our first goal is to find upper bounds for the orthonormal polynomials $p_k.$
Applying Theorem 1.18 of \cite[p. 22]{Lubinskybookexp}, we obtain the estimate
\begin{align*}
\sup_{t\in[-1,1]}|W(a_n t) p_k(a_n t)| \le \sup_{x\in\R}|W(x) p_k(x)| \le C k^{1/6} a_k^{-1/3} \left(T(a_k)/a_k\right)^{1/6},
\end{align*}
that holds uniformly for all $k\ge 1.$ Since $T$ is bounded as a continuous function on $\R$ with finite limit at $\infty$, cf. \eqref{1.3},
we have that
\begin{align} \label{4.3}
\sup_{t\in[-1,1]}|W(a_n t) p_k(a_n t)| \le \sup_{x\in\R}|W(x) p_k(x)| = O\left(k^{1/6} a_k^{-1/2}\right),\quad k\in\N.
\end{align}
Combining the above estimate with Lemma \ref{lem3.1} (for $m=0$), we arrive at
\begin{align} \label{4.4}
W(a_n s) \sup_{z\in B_{\C} (s, 1/n)}|p_k(a_n z)| = O\left(k^{1/6} a_k^{-1/2}\right),\quad k\in\N,
\end{align}
which holds uniformly for $s\in D(\ep).$

We now recall some facts about the reproducing kernel $K_n$. Theorem 1.25 of \cite[p. 26]{Lubinskybookexp} states that
\[
\lim_{n\to\infty} K_n(x,x) W^2(x) = \sigma_n(x),\quad x\in J_n(\ep),
\]
where convergence is uniform in $J_n(\ep)$. Further,  Lemma 5.1(a) of \cite[p. 87]{LL2} gives that uniformly for $x\in J_n(\ep)$,
\[
C_1 \frac{n}{a_n} \le \sigma_n(x) \le C_2 \frac{n}{a_n}.
\]
Hence we have that
\begin{equation}\label{eq:KnW}
C_1 \frac{n}{a_n} \le  K_n(x,x) W^2(x) \le C_2 \frac{n}{a_n},\quad x\in J_n(\ep),
\end{equation}
which implies (after changing variable $x=a_n s$) by \eqref{4.4} that
\begin{equation}\label{4.6}
\frac{|p_k(a_n z)| W(a_n s)}{\sqrt{K_n(a_n s,a_n s) W^2(a_n s)}} \le C \frac{k^{1/6}}{n^{1/2}} \left(\frac{a_n}{a_k}\right)^{1/2} ,\quad s\in D(\ep),\ z\in B_{\C} (s, 1/n),
\end{equation}
where the estimate holds uniformly in $z,s$ and $k\le n\in\N.$ Suppose first that $k \le n^\tau,\ \tau\in(0,1).$ Recall that $a_n$ is an increasing sequence of positive numbers, and that Lemma 3.5(c) of \cite[p. 72]{Lubinskybookexp} states for a constant $C>0$:
\begin{align} \label{eq:ana1}
1\leq \frac{a_n}{a_1}\leq C n^{1/\Lambda} \text{ for all } n\in\N,
\end{align}
where $\Lambda>1$ is from the Definition \ref{def1.1} of $W$. Hence \eqref{4.6} gives that
\begin{equation}\label{4.8}
\frac{|p_k(a_n z)| W(a_n s)}{\sqrt{K_n(a_n s,a_n s) W^2(a_n s)}} \le C n^{(1/\Lambda-1)/2+\tau/6},\quad s\in D(\ep),\ z\in B_{\C} (s, 1/n),
\end{equation}
uniformly for $k\le n^\tau.$ If $n^\tau \le k \le n$ then we use that $K_n(x,x) \ge K_k(x,x)$, together with \eqref{4.4} and \eqref{eq:KnW}, and write instead of \eqref{4.6} that
\begin{align}\label{4.9}
\frac{|p_k(a_n z)| W(a_n s)}{\sqrt{K_n(a_n s,a_n s) W^2(a_n s)}} &\le \frac{|p_k(a_n z)| W(a_n s)}{\sqrt{K_k(a_n s,a_n s) W^2(a_n s)}} \\
&\le C \frac{k^{1/6}}{k^{1/2}} =  C k^{-1/3} \le C n^{-\tau/3},\quad s\in D(\ep),\ z\in B_{\C} (s, 1/n). \nonumber
\end{align}
Selecting the value of $\tau>0$ sufficiently small, so that $b:=\min(\tau/3,(1-1/\Lambda)/2-\tau/6) > 0,$ we obtain from \eqref{4.8}-\eqref{4.9} that
\begin{equation}\label{4.10}
\frac{|p_k(a_n z)| W(a_n s)}{\sqrt{K_n(a_n s,a_n s) W^2(a_n s)}} \le C n^{-b},\quad s\in D(\ep),\ z\in B_{\C} (s, 1/n),
\end{equation}
holds uniformly for all $k\le n.$

Lemma 8.3(a) of \cite{LP} states that  uniformly for $u,v$ in compact subsets of the complex plane, and $x\in J_{n}\left( \varepsilon \right),$
we have as $n\rightarrow \infty ,$
\begin{equation} \label{4.7}
\lim_{n\rightarrow \infty }\frac{K_{n}\left( x+\frac{u}{\tilde{K}_{n}(x,x) },x+\frac{v}{\tilde{K}_{n}(x,x) }\right)} {K_{n}(x,x) }
e^{-\frac{Q'(x)}{\tilde{K}_{n}(x,x)}(u+v)}=\frac{\sin{\pi(v-u)}}{\pi(v-u)},
\end{equation}
where $\tilde{K}_{n}(x,x):=K_n(x,x) W^2(x).$ The original result \eqref{4.7} for the real parameters $u,v$ was established in Theorem 1.2 of \cite{LL3}.
Recall that
\[
|Q'(x)| \le C \frac{n}{a_n}
\]
uniformly in $x\in J_n(\ep),$ by Lemma 3.8(a) of \cite[p. 77]{Lubinskybookexp}. Combining the latter fact with \eqref{eq:KnW}, we obtain that
\begin{equation}\label{4.11}
0 < C_1 \le \frac{Q'(x)}{\tilde{K}_{n}(x,x)} \le C_2
\end{equation}
holds uniformly for $x\in J_n(\ep).$ Thus \eqref{4.7} now gives that
\[
\frac{K_n(x,x)}{|K_n(w,w)|} \le C
\]
uniformly for $x\in J_n(\ep)$ and $w\in B_{\C} (x,a_n/n)$. If we change variables $x=a_n s$ and $w=a_n z$ in the above estimate, and use the result
together with \eqref{4.10} in \eqref{4.1}, then we arrive at the estimate
\begin{equation}\label{eq:delocalization}
\frac{|p_k(a_n z)|}{\sqrt{\sum_{j=0}^{n} |p_j(a_n z)|^{2}}} \le C n^{-b},\quad k=0,1,\ldots,n,
\end{equation}
uniformly for $s\in D(\ep)$ and $z\in B_{\C} (s,1/n).$

\section{Partial proof of Lemma \ref{lm:conditions}: Derivative growth}\label{sec:cond4}
In this section, we prove Condition (C3) as the next step in proving Lemma \ref{lm:conditions}. Both inequalities of (C3) are obtained from the universality limit for the reproducing kernels stated in \eqref{4.7}. Note that the functions under the limit on the right of \eqref{4.7} are entire functions in each complex variable $u$ and $v$ that converge uniformly in $u$ and $v$ on compact subsets of $\C$, and also uniformly for $x\in J_n(\ep) = [-(1-\ep)a_n, (1-\ep)a_n].$ Hence the same uniform convergence  on compact subsets of $\C$ is preserved for all derivatives of these functions with respect to $u$ and $v$. Differentiating \eqref{4.7} and applying induction, we obtain that
\begin{equation} \label{7.1}
\lim_{n\rightarrow \infty } \frac{K_{n}^{(l,m)}\left( x+\frac{u}{\tilde{K}_{n}(x,x) },x+\frac{v}{\tilde{K}_{n}(x,x) }\right)}
{(\tilde{K}_{n}(x,x)^{l+m} K_{n}(x,x) } e^{-\frac{Q'(x)}{\tilde{K}_{n}(x,x)}(u+v)} = F_{l,m}(u,v,x),
\end{equation}
where $F_{l,m}(u,v,x)$ are entire in $u$ and $v$, and the convergence in uniform for $u$ and $v$ in compact subsets of $\C$, and
for $x\in J_n(\ep)$. It follows that each function $F_{l,m}(u,v,x)$ is uniformly bounded for $u$ and $v$ in compact subsets of $\C$, and
for $x\in J_n(\ep)$. Taking into account \eqref{4.11}, we obtain that
\[
\frac{K_{n}^{(l,m)}\left( x+\frac{u}{\tilde{K}_{n}(x,x) },x+\frac{v}{\tilde{K}_{n}(x,x) }\right)}
{(\tilde{K}_{n}(x,x))^{l+m} K_{n}(x,x) }
\]
are bounded uniformly for all $n\in\N$, $u$ and $v$ in compact subsets of $\C$, and $x\in J_n(\ep)$. Recall that  $\tilde{K}_{n}(x,x)=K_n(x,x) W^2(x)$ satisfies \eqref{eq:KnW}, which gives that
\begin{equation} \label{7.2}
\frac{K_{n}^{(l,m)}\left(x+\frac{u a_n}{n},x+\frac{v a_n}{n}\right)}{ K_{n}(x,x) } \le C_{l,m} \left(\frac{n}{a_n}\right)^{l+m},
\end{equation}
uniformly for all $n\in\N$, $u$ and $v$ in compact subsets of $\C$, and $x\in J_n(\ep)$.

To check the first inequality in condition (C3), we use \eqref{7.2} with $l=m=1$, $u=v=0$, and $x=a_n s$, where $s\in D(\ep)$, and estimate
\begin{align*}
\frac{\sum_{j=0}^{n} |p_j'(a_n s)|^{2} a_n^2}{\sum_{j=0}^{n} |p_j(a_n s)|^{2}} = a_n^2 \, \frac{ K_{n}^{(1,1)}\left(a_n s,a_n s\right)}{ K_{n}(a_n s,a_n s) } = O(n^2).
\end{align*}
Thus the first inequality in condition (C3) is true with $c_1=0$.

We now turn to the proof of the second derivative growth condition, where we apply \eqref{7.2} with $l=m=2$, $|u|<1$ and $|v|<1$, and $x=a_n s$, with $s\in D(\ep)$. It follows that
\[
\frac{\dis \sup_{z\in B_{\C} (s, 1/n)} \sum_{j=0}^{n} |p_j''(a_n z)|^2 a_n^4}{\sum_{j=0}^n |p_j(a_n s)|^2} =  a_n^4 \, \frac{\dis \sup_{z\in B_{\C} (s, 1/n)} K_n^{(2,2)}(a_n z,a_n \bar{z})}{K_n(a_n s,a_n s)} = O(n^4).
\]
We also used that the coefficients of the orthogonal polynomials $p_j$ are real, to express the numerator in the above equation via the kernel $K_n^{(2,2)}.$ This completes verification of the derivative growth conditions with $c_1=0$.

\section{Final piece of proof of Lemma \ref{lm:conditions}: Anti-concentration}\label{sec:cond1}
In this section, we prove Condition (C4) as the last step of proving Lemma \ref{lm:conditions}. In particular, we show that the minimum of $\ps$ in a local ball is bounded away from 0 with high probability.
The key result of this section is
\begin{lemma} [Anti-concentration]\label{cond-smallball}  For any positive constants $A, c$ and $c_1$,  the following holds: for every $x_0\in [-1+\ep, 1-\ep]$, with probability at least $1 - O(n^{-A})$, there exists $x_1\in [x_0-c/n, x_0+c/n]$ for which $|W(a_n x_1)\ps(x_1)|\ge \exp(-n^{c_1})$.
\end{lemma}

We will prove an actually stronger result below
\begin{lemma}\label{lm:cond-smallball'}  For any positive constants $A, c$ and $c_1$,  the following holds: for every $x_0\in [-1+\ep, 1-\ep]$, there exists $x_1\in [x_0-c/n, x_0+c/n]$ such that
	$$\P\left(|W(a_n x_1)\ps(x_1)|\le \exp(-n^{c_1})\right) < n^{-A}.$$
\end{lemma}

The rest of this section is devoted to the proof of this lemma. Before going into the detail, let us first provide a road map.

\begin{enumerate}
	\item Step 0: Assume by contradiction that Lemma \ref{lm:cond-smallball'} fails.
	\item Step 1: We recall the classical three-term recurrence for three consecutive $p_m$. From that, we can write $p_{m+k}$ in terms of $p_m$ and $p_{m-1}$ which we call a jumped three-term recurrence (see Lemma \ref{lm:jump3}).
	
	\item Step 2: From the inverse Littlewood-Offord theory, we know that for a given $x$, for the function $W(a_n x)\ps(x) = \sum_{i=0}^{n} \xi_i W(a_n x) p_i(a_nx)$ to satisfy
		$$\P(|W(a_n x)\ps(x)|\le \exp(-n^{c_1})) \ge n^{-A},$$
		it is necessary that the components $W(a_n x) p_i(a_n x)$ are close to a general arithmetic progression. Roughly speaking, this means that they have a low linear-dimension. See Lemma \ref{lm:n'}.
		
	\item Step 3: From Step 2, we know that there are many non-trivial linear combinations of the $(W(a_n x) p_i(a_n x))$ that are very small
	$$\sum_{j=1}^{O(1)} W(a_n x) p_{i_j}(a_nx) \approx 0 \text{ (which precisely means $  \exp(-n^{c})$-close to 0}).$$
	Using the three-term recurrence in Step 1 to write all of these $p_{i_j}$ as combination of $p_m$ and $p_{m-1}$, we obtain a so-called two-term degeneracy
	$$W(a_n x) p_m(a_n x) \approx \Gamma(x)  W(a_n x) p_{m-1}(a_n x)$$
	for some low-degree fractional polynomial $\Gamma$. See Lemma \ref{lm:2deg}.
	
	\item Step 4: Adapting the same strategy from Step 3 (using two-term degeneracy in place of three-term recurrence), we can reduce two-term degeneracy to (one-term) degeneracy. See Lemma \ref{lm:1deg}.
	
	\item Step 5: Show that degeneracy is impossible as it reduces to a polynomial of low degree and large leading term taking small value on a large interval.
	
\end{enumerate}

We shall now go into the details.
\subsection{Step 1: Jumped three-term recurrence}
The following lemma covers a classic fact concerning the three-term recurrence for $(p_m)$.
\begin{lemma}\label{lm:AmBm}  The polynomials $p_m$ satisfy a three-term recurrence relation
\begin{equation}\label{eqn:rec}
	xp_{m}=A_{m}p_{m+1}+B_{m}p_{m}+A_{m-1}p_{m-1}
\end{equation}
Moreover, there exists $r\in \left( 0,1\right) $ and $C_{1},C_{2}>0$ such that for $%
m\geq 1,$%
\begin{eqnarray*}
	C_{1} &\leq &A_{m}\leq C_{2}m^{r}; \\
	\left\vert B_{m}\right\vert  &\leq &C_{2}m^{r}.
\end{eqnarray*}%
\end{lemma}
For our purpose, we only need a weaker statement $m^{-C}\ll A_m \ll m^{C}\   \text{ and }\  |B_m|\ll m^{C}.$
\begin{proof}[Proof of Lemma \ref{lm:AmBm}]
We let $a_{\pm m}$ denote the Mhaskar-Rakhmanov-Saff numbers, so that
\begin{eqnarray*}
	m &=&\frac{1}{\pi }\int_{a_{-m}}^{a_{m}}\frac{xQ^{\prime }\left( x\right) }{%
		\sqrt{\left( x-a_{-m}\right) \left( a_{m}-x\right) }}dx; \\
	0 &=&\frac{1}{\pi }\int_{a_{-m}}^{a_{m}}\frac{Q^{\prime }\left( x\right) }{%
		\sqrt{\left( x-a_{-m}\right) \left( a_{m}-x\right) }}dx.
\end{eqnarray*}%
We let
\begin{eqnarray*}
	\delta _{m} &=&\frac{1}{2}\left( a_{m}+\left\vert a_{-m}\right\vert \right) ;
	\\
	\beta _{m} &=&\frac{1}{2}\left( a_{m}+a_{-m}\right) ,
\end{eqnarray*}%
so that
\[
\left[ a_{-m},a_{m}\right] =\left[ \beta _{m}-\delta _{m},\beta _{m}+\delta
_{m}\right] .
\]%
is the $m$th Mhaskar-Rakhmanov-Saff interval. \newline

 From \cite[Theorem 15.2, p. 402]{Lubinskybookexp}, which holds for a larger class of
weights,
\[
\lim_{m\rightarrow \infty }\frac{A_{m}}{\delta _{m}}=\frac{1}{2}\text{ and }%
\lim_{m\rightarrow \infty }\frac{B_{m}-\beta _{m}}{\delta _{m}}=0.
\]%
Note that by definition,  $\left\vert \beta _{m}\right\vert \leq \delta _{m}$%
. So all we need to prove is that there exists $r<1$ such that both
\[
\delta _{m}=O\left( m^{r}\right) \text{.}
\]%
This follows from \cite[Lemma 3.4(c), p. 72]{Lubinskybookexp}, where it is proved that
there exists $r_{0}$ such that for $t\geq r\geq r_{0},$%
\[
1\leq \frac{\delta _{t}}{\delta _{r}}\leq C\left( \frac{t}{r}\right)
^{1/\Lambda }.
\]%
Since $\Lambda >1$, we can take $r=1/\Lambda $.
\end{proof}

From the recurrence, we obtain that
\begin{equation*}\label{key}
 p_{m+1} = \left (\frac{1}{A_m} x+\dots\right )p_m + \frac{A_{m-1}}{A_m}p_{m-1},
\end{equation*}
\begin{eqnarray*}\label{key}
	p_{m+2} &=& \left (\frac{1}{A_{m+1}} x+\dots\right )p_{m+1} + \frac{A_{m-1}}{A_m}p_{m}\notag\\
	 &=& \left (\frac{1}{A_{m+1}A_{m}} x^2+\dots\right )p_{m} + \left (\frac{A_{m-1}}{A_{m+1}A_{m}} x+\dots\right )p_{m-1},
\end{eqnarray*}
and so, by the bounds from Lemma \ref{lm:AmBm}, we obtain the following ``jump" recurrence.
\begin{lemma}\label{lm:jump3}
For all $k, m\ge 1$,
 \begin{eqnarray}\label{eq:pmk:2}
 	p_{m+k}  	&=& \left (\frac{1}{\prod_{i=m}^{m+k-1}A_{i}} x^k+\dots\right )p_{m} + \left (\frac{A_{m-1}}{\prod_{i=m}^{m+k-1}A_{i}} x^{k-1}+\dots\right )p_{m-1},
 \end{eqnarray}
where in the dots are the lower order terms whose coefficients are $O(n^{O_k(1)})$.
\end{lemma}

\subsection{Step 2: Large anti-concentration probability implies arithmetic structure}
Let $$\beta= \exp(-n^{c_1}).$$
 Let $x\in  [x_0-c/n, x_0+c/n]$ and assume otherwise that $\P(|W(a_n x)\ps(x)|\ge \beta) > n^{-A}$. Then by \cite[Theorem 2.9]{HNgV} (see also \cite{NgVsurv, TV-LO}), there exists an exceptional set $B_x \subset [n]$ of size at most  $n^{\delta/2}$ so that the set $\{|W(a_n x)p_i(a_n x), i\notin B_x\}$ are $\beta$-closed to a GAP of size $n^{O_A(1)}$ and rank at most $r'=O_A(1)$. In other words, there exist a set of $r'$ generators $g_1,\dots, g_{r'}$ (all depending on $x$) such that for all $i \notin B_x$, there exists integers $k_{i1},\dots, k_{ir'}$, all of order $n^{O_A(1)}$ satisfying
\begin{equation}\label{GAP:approx}
|W(a_n x)p_i(a_n x) - \sum_{j=1}^{r'} k_{ij} g_j| =O(\beta).
\end{equation}
	
In the following lemma, by passing to a subset, we can reduce to the case that the rank $r'$ and the coefficients $k_{ij}$ are the same. By the delocalization bound in Section \ref{sec:cond3}, we can also reduce to a subset of induces $n'$ at which $|W(a_n x)p_{n'}(a_n x)| $ is not too small.
 	
\begin{lemma}\label{lm:n'} There exist constants $B, \delta>0$ such that for any constant $K$, there exist a measurable set $\CE \subset [x_0-c/n, x_0+c/n]$ and an index $n^{\delta} < n' < n-n^{\delta}$ such that
	\begin{enumerate}
		\item $\lambda_{Leb}(\CE)>2c/n^2$;
		\vskip .1in
		\item $|W(a_n x)p_{n'}(a_n x)| \ge n^{-B}, x\in \CE$;
		\vskip .1in
		\item $[n', n'+K]\cap B_x=\emptyset, x\in \CE$.
		\vskip .1in
		\item The GAP ranks $r'$ are identical for all $x\in \CE$. Moreover, the coefficients $k_{ij}(x) = k_{ij}(y)$ for all $n'\le i\le n'+K$, $1\le j\le r'$, $x, y\in \CE$.
	\end{enumerate}
Here, $\lambda_{Leb}$ is the Lebesgue measure on $\R$.
\end{lemma}	
	In this section we will choose $K$ sufficiently large compared to $r$ and $A$, but still $K=O(1)$.
	
	\begin{proof}
		We start with the second item. For $x\in   [x_0-c/n, x_0+c/n]$, by \eqref{eq:delocalization}, for some constant $b>0$ and for all $0\le i\le n$,
		$$\frac{W(a_n x) |p_i(a_n x)|}{W(a_n x)\sqrt{K_n(x, x)}} = \frac{W(a_n x) |p_i(a_n x)|}{W(a_n x)\sqrt{\sum_{j=0}^{n} p_j^{2}(a_n x)}} \le C n^{-b}.
		$$
		Thus, for $\delta = b/2$, the contribution of $i\in [n^{\delta},  n-n^{\delta}]$ is dominant
			By \eqref{eq:KnW} and \eqref{eq:ana1}, for some constant $\Lambda>1$,
		$$K_n(x, x) W^2(a_n x)\gg \frac{n}{a_n}\gg \frac{n}{n^{1/\Lambda}} = n^{1-1/\Lambda}.$$
		Thus,
		$$\sum_{i=n^{\delta}}^{n-n^{\delta}} W(a_n x) |p_i(a_n x)|\ge \frac12.$$
		
	So, there exists $B>0$ sufficiently large so that for each $x$, the set $N_x$ of indices $i$ where $n^{\delta} < i < n-n^{\delta}$ and $|W(a_n x)p_i(a_n x)| > n^{-B}$ satisfies
		$$|N_x| \ge n^{\delta}.$$
		
		For each $x\in  [x_0-c/n, x_0+c/n]$, because the set of $i$ for which $\{i,\dots, i+K\}\cap B_x \neq \emptyset$ has size at most $2K n^{\delta/2}$, and because $|N_x| \ge n^{\delta}$, there exists $i \in N_x$ for which  $\{i,\dots, i+K\}\cap B_x = \emptyset$. For each such $x$, we fix such an index $i=i(x)$ (note that by continuity, for $x$ very close to $x'$ we can choose $i(x)=i(x')$). When $x$ varies over $[x_0-c/n, x_0+c/n]$, by pigeonhole principle, there exists a measurable set $\CE'$ of measure $\lambda_{Leb}(\CE') \ge (2c/n)/n$ where $i(x)$ takes the same value for all $x\in \CE'$. Set $n'$ to be this common value, we complete the proof.

	Recall that by the application of \cite[Theorem 2.9]{HNgV} above, we can identify $W(a_n x)p_{i}(a_n x), i \in \{n', n'+1, \dots, n'+K\}$ with a vector $\Bk_i=\Bk_i(x) = (k_{i1},\dots, k_{i r'})\in \mathbb Z^{r'}$ with $|k_{i r'}| \le n^{O_A(1)}$ and $r'=O(1)$. More importantly, as there are only $n^{O_A(1)}$ such choices of vectors for $n' \le i\le n'+K$, by passing to a subset $\CE \subset \CE'$, still of measure $n^{-\Omega_{K,A}(1)}$, we can assume that these vectors $\Bk_i(x)$ are all the same for all $x\in \CE$.
\end{proof}
	
In the next step, we show that there are many $n''$ in the range $n' \le n'' \le n'+K$ such that $|W(a_n x)p_{n''}(a_n x)| \ge n^{-B}, x\in \CE$.

 \begin{claim}\label{claim:non-degenerate} Let $B'$ be sufficiently large depending on $B$ and $K$. Then for all $0\le i\le K$, we must have either $|W(a_n x)p_{n'+i}(a_n x)| \ge n^{-B'}$ or  $|W(a_n x)p_{n'+i+1}(a_n x)| \ge n^{-B'}$ for each $x\in \CE$.
 \end{claim}
\begin{proof} Assume otherwise, we can use the recurrence \eqref{eqn:rec} to go backward and use the lower bound for $A_m$ from Lemma \ref{lm:AmBm} to arrive at $|W(a_n x)p_{n'}(a_n x)| \le n^{-B}$, a contradiction.
\end{proof}

\subsection{Step 3: Arithmetic structure  and jumped three-term recurrence imply 2-term degeneracy}
The following lemma shows the so-called 2-term degeneracy of the form
$$W(a_n x) p_{n_0+1}(a_n x) \approx  W(a_n x) p_{n_0} (a_n x) $$
for a large number of choices for $n_0$.
\begin{lemma}\label{lm:2deg} Let $K$ be an arbitrarily large constant and $\ep_1$ be a sufficiently small constant. Let $n'$ be as in Lemma \ref{lm:n'}. There exists a sub-interval $[n_1, n_2]\subset [n', n'+K]$ of length $n_2-n_1\gg \log K$ such that for all $n_0\in [n_1, n_2]$, we have
	\begin{eqnarray}
	W(a_n x) p_{n_0+1}(a_n x)&=&- \frac{\frac{A_{n_0}}{\prod_{i=n_0+1}^{n_0+r-1}A_{i}} x^{r-2}+\dots  }{  \frac{1}{\prod_{i=n_0+1}^{n_0+r-1}A_{i}} x^{r-1}+\dots }  W(a_n x) p_{n_0} (a_n x)   +O(n^{O(1)} \beta^{1-\ep_1}), \quad x\in \CE. \notag
	\end{eqnarray}
\end{lemma}
Here, we note that the set $\CE$ may be a subset of the $\CE$ in Lemma \ref{lm:n'} but it always has a Lebesgue measure of $\Omega(n^{-2})$.

\begin{proof}[Proof of Lemma \ref{lm:2deg}]
We consider the vectors $\Bk_{n'}, \dots, \Bk_{n'+K}$. For each $n'\le i\le j\le n'+K$, we let
$$G_{[i,j]} = \{ \Bk_{i},\dots, \Bk_{j}\}$$
and $H_{[i,j]}$ be the subspace (over $\R$) spanned by the vectors in $G_{[ij]}$. Set
$$s_0 = \lfloor (\log_2 (K/2))/2 \rfloor.$$
For $0\le s \le s_0$, consider the nested subsets  $G_{[n'+K/2- 2^{s_0+s}, n'+K/2+ 2^{s_0+s}]}$. As the dimensions $\dim(H_{[n'+K/2- 2^{s_0+s}, n'+K/2+ 2^{s_0+s}]})$ are non-decreasing in $s$, and all bounded by $r'$, by pigeon-hole principle, provided that $K$ is sufficiently large given $r'$, there exists $s_1 \le  s_0 -1$ such that
\begin{equation}\label{eqn:rankeq}
\dim(H_{[n'+K/2- 2^{s_0+s_1}, n'+K/2+ 2^{s_0+s_1}]})= \dim(H_{[n'+K/2- 2^{s_0+s_1+1}, n'+K/2+ 2^{s_0+s_1+1}]}).
\end{equation}
Let $n_1=n'+K/2 -2^{s_0+s_1+1}, n_1' = n'+K/2+ 2^{s_0+s_1+1}$, $n_2=n'+K/2 -2^{s_0+s_1}$, and $n_2'=n'+K/2+ 2^{s_0+s_1}$ be the endpoints.

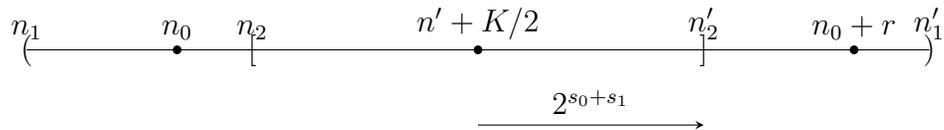
\begin{figure}[h!]\centering
	
	\begin{tikzpicture}
	
	\draw  (0,2) node{(} node[above]{$n_1$}   -- (2, 2)node[fill,circle,scale=0.3]{}node[above]{$n_0$}  -- (3,2) node{[}node[above]{$n_2$}  -- (6,2) node[fill,circle,scale=0.3]{} node[above]{$n'+K/2$} -- (9,2) node{]}node[above]{$n_2'$} --(11,2) node[fill,circle,scale=0.3]{}node[above]{$n_0+r$}-- (12,2) node{)}node[above]{$n_1'$};
	\draw [-stealth] (6,1.0) --  (7.5,1.0)  node[above]{$2^{s_0+s_1}$} -- (9,1.0);
	\end{tikzpicture}
	
	\label{figure:swap}
	
\end{figure}

Hence the subspaces $H_{[i,j]}$ where $n_1 \le i \le n_2$ and $n'_2\le j \le n'_1$  are all the same.
In what follows, we set
$$r= 3 \times 2^{s_0+s_1}.$$
Choose $n_0$ be any index in $[n_1, n_2]$. In this range, by \eqref{eqn:rankeq} we have
$$\Bk_{n_0+r} \in \langle \Bk_{n_0+i}, 1\le i\le r-1 \rangle,$$
the linear vector space generated by $\Bk_{n_0+i}$ over $\R$.

So we have from \eqref{GAP:approx} that
\begin{equation}\label{eqn:GAPid}
W(a_n x) p_{n_0+r}(a_n x) = \sum_{1\le i\le r-1} c_{ik} W(a_n x) p_{n_0+i}(a_n x) + O(n^{O(1)}\beta), x\in \CE
\end{equation}
where $c_{ik}$ are rational of heights bounded by $n^{O(1)}$.

By the jump three-term recurrence \eqref{eq:pmk:2}, we have
$$p_{n_0+r} =\left (\frac{1}{\prod_{i=n_0+1}^{n_0+r-1}A_{i}} x^{r-1}+\dots\right )p_{n_0+1} + \left (\frac{A_{n_0}}{\prod_{i=n_0+1}^{n_0+r-1}A_{i}} x^{r-2}+\dots\right )  p_{n_0}$$
and
$$\sum_{1\le i\le r-1} c_{ik}  p_{n_0+i}  = (\text{polynomial of degree }\le r-2)p_{n_0+1} + (\text{polynomial of degree }\le r-3) p_{n_0}.$$
Hence,
   $$\left (\frac{1}{\prod_{i=n_0+1}^{n_0+r-1}A_{i}} x^{r-1}+\dots\right ) W(a_n x) p_{n_0+1}(a_n x)+{ \left (\frac{A_{n_0}}{\prod_{i=n_0+1}^{n_0+r-1}A_{i}} x^{r-2}+\dots\right ) }  W(a_n x) p_{n_0} (a_n x)$$
\begin{equation}
 \qquad =O(n^{O(1)}\beta), \quad x\in \CE.\label{eq:n0:to}
\end{equation}
Let us denote by $Q_{n_0}(x)$ the polynomial in front of $W(a_n x) p_{n_0+1}(a_n x)$. Note that $Q_{n_0}$ is just a combination (with at most $r=O(1)$ terms and rational coefficients $c_{ik}$ of heights bounded by $n^{O(1)}$) of fixed polynomials appeared in \eqref{eq:pmk:2} (which, given $\mu$, consists of $\le n^2$ polynomials as there are $\le n^2$ choices for $m$ and $k$ there). So, the number of possible choice for $Q_{n_0}$ (as $n_0, r, c_{ik}$ vary) is at most $n^{O(1)}$.

The following lemma controls the set of points at which $Q_{n_0}$ is small. We defer the simple proof (by relying on Remez inequality) to the end of this subsection.
\begin{lemma}\label{lm:S} Let $S(x) = a_d x^{d}+a_{d-1}x^{d-1}+\dots+a_0$ be a polynomial with degree $d=O(1)$ satisfying $|a_d|\ge n^{-B}$ and $|a_i|\le n^{B}$ for all $0\le i\le d$ where $B$ is some large constant. Then the maximum value of $S$ on $I_B = [-n^{3B}, n^{3B}]$ is $(d+1) n^{B} n^{3Bd}$. Moreover, for all sufficiently small  $\ep_1>0$ the set of all $x\in I_B$ such that $|S(x)|\le \beta^{\ep_1}$ has measure $O(n^{3B} \beta^{\ep_1/d})$.
\end{lemma}

We can take any small constant $\ep_1\le \frac12$ that satisfies this lemma.

Taking union over all $n^{O(1)}$choices of $Q_{n_0}$, the set of points $x$ at which $|Q_{n_0}(x)|\le  \beta^{\ep_1}$ has measure at most $\beta^{2\ep_1/r}$. So, by removing these $x$ from $\CE$, the new set which we again denote by $\CE$ still has size $n^{-\Omega(1)}$. On this set, on \eqref{eq:n0:to}, by dividing by $Q_{n_0}(x)$, we get
\begin{eqnarray}
 W(a_n x) p_{n_0+1}(a_n x)&=&- \frac{\frac{A_{n_0}}{\prod_{i=n_0+1}^{n_0+r-1}A_{i}} x^{r-2}+\dots  }{  \frac{1}{\prod_{i=n_0+1}^{n_0+r-1}A_{i}} x^{r-1}+\dots }  W(a_n x) p_{n_0} (a_n x)   +O(n^{O(1)} \beta^{1-\ep_1}), \quad x\in \CE. \notag
\end{eqnarray}
This finishes the proof of Lemma \ref{lm:2deg}.
\end{proof}

\begin{proof}[Proof of Lemma \ref{lm:S}]
	For the first part, we have
	$$|S(x)|\le \sum_{i=0}^{d} |a_i|x^{i} \le (d+1) n^{B} n^{3Bd}.$$
	For the second part, let
	$$E = \{x\in I_B: |S(x)|\le \beta^{\ep_1}\}.$$
	To bound the size of $E$, we shall use the classical Remez inequality: for every measurable subset $E\subset J$,  we have
	$$
	\max_{I_B}|S|\leq \left (\frac{4|I_B|}{|E|}\right )^{d}\sup_{E}|S|\le \left (\frac{4|I_B|}{|E|}\right )^{d}\beta^{\ep_1}.
	$$
	Since
	$$\max_{I_B}|S|\ge |S(n^{3B})| \ge |a_d|n^{3Bd}-\sum_{i=0}^{d-1} |a_i| n^{3Bi}\ge n^{3Bd-B} - dn^{3B(d-1)+B}\ge 1$$
	for sufficiently large $n$. Thus,
	$$|E|\le 8n^{3B} \beta^{\ep_1/d}.$$
	This completes the proof of Lemma \ref{lm:S}.
\end{proof}

\subsection{Step 4: Arithmetic structure and  two-term degeneracy imply degeneracy}
We shall now derive a (one-term) degeneracy of the form
$$W(a_n x) p_{n_1'} (a_n x) \approx 0.$$
\begin{lemma}\label{lm:1deg} Let $[n_1, n_2]$ be as in Lemma \ref{lm:2deg}. There exists $n_3\in [n_1, n_2]$ such that
	$$\frac{R_1(x)}{R_2(x)}  W(a_n x) p_{n_3} (a_n x) = O(\beta^{1-\ep_1}), \quad x\in \CE$$
	where $R_1$ and $R_2$ are polynomials satisfying the hypothesis of Lemma \ref{lm:S} and $\deg(R_1) = \deg(R_2)-1$.
\end{lemma}

\begin{proof}[Proof of Lemma \ref{lm:1deg}] We continue using the notations in the proof of Lemma \ref{lm:2deg}.
Let $r_0 = n_2- n_1-1$. Consider the vectors $\Bk_{n_1+i}, 0\le i\le r_0$. Because the vectors are in $\R^{r'}$ and $r_0$ is much large than $r'$, there must exist $k$ such that
$$\Bk_{n_1+2k}, \Bk_{n_1+2k+1}  \in \langle  \Bk_{n_1+2k+1+i}, 1 \le i \le r_0-2k-1 \rangle.$$
On the other hand, we have learned from Claim \ref{claim:non-degenerate} that either $|W(a_n x) p_{n_1+2k} (a_n x)| \ge n^{-B'}$ or $|W(a_n x) p_{n_1+2k_1} (a_n x)| \ge n^{-B'}$. Assume without loss of generality that $|W(a_n x) p_{n_1+2k} (a_n x)| \ge n^{-B'}$. Set
$$n_3= n_1+2k.$$

As $\Bk_{n_3} \in \langle  \Bk_{n_3+i}, 0\le  i \le r_0-2k \rangle$, from \eqref{GAP:approx}, we can find rational coefficients $d_{\ell}$ of height at most $n^{O(1)}$ such that
\begin{eqnarray}\label{eqn:1r_0}
W(a_n x) p_{n_3+1}(a_n x) - \sum_{\ell=2}^{r_0-2k} d_{\ell} W(a_n x) p_{n_3+\ell}(a_n x)= O(n^{O(1)}\beta), \quad x\in \CE.
\end{eqnarray}
 
Recall from Lemma \ref{lm:2deg} that for all $j=1, \dots, r_0-1$,
\begin{eqnarray}
\quad \ W(a_n x) p_{n_1+j+1}(a_n x)&=&- \frac{\frac{A_{n_1+j}}{\prod_{i=n_1+j+1}^{n_1+j+r-1}A_{i}} x^{r-2}+\dots  }{  \frac{1}{\prod_{i=n_1+j+1}^{n_1+j+r-1}A_{i}} x^{r-1}+\dots }  W(a_n x) p_{n_1+j} (a_n x)   +O(n^{O(1)}  \beta^{1-\ep_1}). \label{eqn:j+1:j}
\end{eqnarray}

Applying \eqref{eqn:j+1:j} for $j=n_3-n_1$, and then recursively for $j=n_3-n_1+1, \dots, n_3-n_1+ r_0-2k$ and all $x\in \CE$,
\begin{eqnarray}
W(a_n x) p_{n_3+s+1}(a_n x)&=&- \frac{\frac{\prod_{\ell=0}^{s}A_{n_3+\ell}}{\prod_{\ell=0}^{s}\prod_{i=n_3+\ell+1}^{n_3+\ell+r-1}A_{i}} x^{(s+1)(r-2)}+\dots  }{  \frac{1}{\prod_{\ell=0}^{s}\prod_{i=n_3+\ell+1}^{n_3+\ell+r-1}A_{i}} x^{(s+1)(r-1)}+\dots }  W(a_n x) p_{n_3} (a_n x)   +O( n^{O(1)} \beta^{1-\ep_1}) \notag\\
&=:& - \frac{ \bar A_{s+1} x^{(s+1)(r-2)}+\dots  }{ \hat A_{s+1} x^{(s+1)(r-1)}+\dots }  W(a_n x) p_{n_3} (a_n x)   +O(n^{O(1)} \beta^{1-\ep_1}) \notag\\
&=:& \Gamma_{s+1}(x)  W(a_n x) p_{n_3} (a_n x)   +O(n^{O(1)} \beta^{1-\ep_1}) \notag.
\end{eqnarray}

Substitute into \eqref{eqn:1r_0} we obtain
$$R(x)  W(a_n x) p_{n_3} (a_n x) = O(n^{O(1)} \beta^{1-\ep_1}), \quad x\in \CE$$
where
$$R(x) =  \Gamma_{1}(x)-\sum_{\ell=2}^{r_0-2k} d_\ell \Gamma_{\ell}(x).$$
Since $\Gamma_{\ell}(x)$ is a fraction of two polynomials with the numerator's degree is $\ell$ less than the denominator's degree, $R(x)$ can be written as $R(x) = \frac{R_1(x)}{R_2(x)}$ where $R_2(x)$ is the product of all denominators of $\Gamma_{\ell}(x)$ and $\deg(R_1) = \deg(R_2)-1$. The leading coefficient of $R_1$ is
$$ \bar A_{1} \hat A_{2} \dots \hat A_{r_0-2k} \ge n^{-\Omega(1)}.$$
Moreover, all coefficients of $R_1$ are $O(n^{O(1)})$ by Lemma \ref{lm:AmBm}. So, $R_1$ satisfies the hypothesis of Lemma \ref{lm:S}. The same reasoning applies for $R_2$.
This completes the proof of Lemma \ref{lm:1deg}.
\end{proof}

\subsection{Step 5: Degeneracy is impossible} We are now ready to derive a contradiction which completes the proof of Lemma \ref{lm:cond-smallball'}.
Since $|W(a_n x) p_{n_3} (a_n x)| \ge n^{-B}$ by the definition of $\CE'$, we get
$$R(x) = O(n^{O(1)} \beta^{1-\ep_1}), \quad  x\in \CE.$$
Since $R_2$ satisfies the hypothesis of Lemma \ref{lm:S}, we get that for all $x\in \CE$,
$$|R_2(x)|\ge n^{O(1)} \beta^{1-\ep_1}.$$
So,
$$|R_1(x)|= O(n^{O(1)} \beta^{1-\ep_1}), \quad x\in \CE.$$

Since $R_1$ also satisfies the hypothesis of Lemma \ref{lm:S} and since $|\CE|\ge n^{-O(1)}$ much larger than $n^{O(1)} \beta^{1-\ep_1}$, we obtain a contradiction to Lemma \ref{lm:S}. This completes the proof of Lemma \ref{lm:cond-smallball'}.

\section{Proof of Theorem \ref{thm:local:expectation}} \label{sec:proof:local}
Our idea is to apply Theorem \ref{thm:general:real} to the indicator function $\textbf{1}_{[a, b]}$. However, since the indicator function is not smooth, we will approximate it below and above by smooth functions, say $F_1$ and $F_2$. To this end, let $\alpha = n^{-1-c/12}$ where $c$ is the constant in Theorem \ref{thm:general:real} and let $F_1, F_2:\R \to \R$ be any continuous functions that are continuously differentiable up to degree 6 and satisfy the following
\begin{itemize}
	\item $\textbf{1}_{[a-\alpha, b+\alpha]}\ge F_2\ge \textbf{1}_{[a, b]}\ge F_1\ge  \textbf{1}_{[a+\alpha, b-\alpha]}$ \quad pointwise,
	\item $||F^{a}_i||_{\infty}\ll \alpha^{-a} \ll n^{a+c/2}$ for all $i=1, 2$ and $1\le a\le 6$.
\end{itemize}
Applying Theorem \ref{thm:general:real} to $G_i = n^{-c/2} F_i$ (where the factor $n^{-c/2}$ is added so that $||G^{a}_i||_{\infty}\ll n^{-a}$ as in the hypothesis), we obtain
\begin{eqnarray}
\E\sum G_1\left (\zeta_{i} \right)
-\E\sum G_1\left (\tilde \zeta_{i} \right) \ll n^{-c},\quad \E\sum G_2\left (\zeta_{i} \right)
-\E\sum G_2\left (\tilde \zeta_{i} \right) \ll n^{-c}\nonumber
\end{eqnarray}
which gives
\begin{eqnarray} \label{eq:F:12:0}
\E\sum F_1\left (\zeta_{i} \right)
-\E\sum F_1\left (\tilde \zeta_{i} \right) \ll n^{-c/2}, \quad \E\sum F_2\left (\zeta_{i} \right)
-\E\sum F_2\left (\tilde \zeta_{i} \right) \ll n^{-c/2}.
\end{eqnarray}
We shall prove later that
\begin{equation} \label{eq:F:12}
\E\sum \tilde F_1 \left (\zeta_{i} \right)  - \E\sum \tilde F_2\left (\zeta_{i} \right) \ll n^{-c/24}.
\end{equation}
Since $\sum F_1 \left (\zeta_{i} \right)  \le N[a, b]\le \sum F_2\left (\zeta_{i} \right)  $,  we conclude from \eqref{eq:F:12:0} and \eqref{eq:F:12} that
\begin{equation}
\E N_{\ps}[a, b] - \E N_{\wps}[a, b] \ll n^{-c/24}\nonumber
\end{equation}
as claimed.

To prove \eqref{eq:F:12}, we note that $0\le F_2 - F_1\le \textbf{1}_{[a-\alpha, a+\alpha]} + \textbf{1}_{[b-\alpha, b+\alpha]} $, thus it suffices to show the following
\begin{lemma}
	Let $I\subset D(\ep)$ be an interval of length $n^{-1-c/12}$. Then
	\begin{equation*}\label{key}
	\E N_{\wps}(I) \ll n^{-c/24}.
	\end{equation*}
\end{lemma}
\begin{proof}
	By the Kac-Rice formula \cite{Kac1943average}, we have
	\begin{eqnarray}\label{key}
	\frac{1}{n}\E N_{\wps}(I) &=& \frac{1}{n}\E N_{\tilde P_n}(a_n I) = \int_{a_nI} \frac{1}{\pi} \frac{K_{n}(x, x)}{n}\sqrt{\frac{K_{n}^{(1,1)}(x, x)}{K_{n}(x, x)^{3}} - \left (\frac{K_{n}^{(0,1)}(x, x)}{K_{n}(x, x)^{2}}\right )^{2}} dx  \nonumber\\
	&\le&\int_{a_nI} \frac{1}{n\pi}  \sqrt{\frac{K_{n}^{(1,1)}(x, x)}{K_{n}(x, x)}  } dx =\int_{a_nI} \frac{1}{n\pi}  \sqrt{\frac{K_{n}^{(1,1)}(x, x) W^{2}(x)}{K_{n}(x, x)W^{2}(x)}  } dx \nonumber.
	\end{eqnarray}
 	
 	By \eqref{eq:KnW}, we have
	\begin{equation*}\label{key}
	\inf_{x\in[-1+\ep, 1-\ep]}  K_{n}(x, x)W^{2}(x) \gg  \frac{n}{a_n}.
	\end{equation*}
	By \eqref{eq:KnW} and \eqref{7.2} (with $u=v=0$), we have
	\begin{equation*}\label{key}
	\sup_{x\in [-1+\ep, 1-\ep]}  K^{(1, 1)}_{n}(x, x) W^{2}(x) \ll \frac{n^{3}}{a_n^{3}}.
	\end{equation*}
	Therefore,
	\begin{eqnarray*}\label{key}
		\frac{1}{n}\E N_{\wps}(I) &\ll&  \int_{a_nI}  \frac{1}{a_n} dx \ll |I| = n^{-1-c/12}
	\end{eqnarray*}
	as desired.
\end{proof}

\section{Proof of Theorem \ref{thm:general:real}} \label{sec:proof:uni}
The proof roughly follows \cite[Theorem 2.6]{nguyenvurandomfunction17} with our $\bf F_n$ plays the role of the $F_n$ in \cite{nguyenvurandomfunction17}.

 The only issue with our $F_n$ is that it is not analytic. And there are two places in \cite{nguyenvurandomfunction17} that use analyticity. We shall explain how we walk around these two places.

In the first place,  is to derive an alternative verstion of \cite[Equation 11]{nguyenvurandomfunction17} where the use of the  Green's formula is no longer applicable to our $F_n$ because it is not necessarily analytic. Firstly, we observe that the roots of $P_n^*$ and $F_n$ are the same. Thus, for any smooth function $G_j$ on $\C$, we get the following equations which can be used as a replacement of \cite[Equation 11]{nguyenvurandomfunction17}
\begin{eqnarray*}
\sum_{i} G_j({\zeta}_i)&=& \int_{\mathbb C}\log |P^*_n(z)|H_j(z)dz \\
&=& \int_{B( z_j, 1/10)}\log |F_n(u_j)|H_j(u_j)du_j - c(G),
\end{eqnarray*}
where $H_j(z) = \frac{1}{2\pi}\triangle G_j(z)$ and
$$c(G) = \int_{B( z_j, 1/10)}\log |W(a_n u_j)|H_j(u_j)du_j$$
is a deterministic number.

 The second place, also more crucial, is the treatment of \cite[Lemma 8.2]{nguyenvurandomfunction17} which originally requires that $F_n$ is analytic. Here, we restate the statement for our $F_n$ which is not necessarily analytic.
\begin{lemma}\label{lm:2norm}
	Let $0<c_2<1$ and let $F_n$ be a function of the form $F_n(x) = P_n^*(x) W(a_n x)$ where $P_n^*$ is an entire function and $-\log W =Q$ satisfies \eqref{eq:Q}. Assume that $|F_n(w)|\ge \exp(-n^{c_2})$ for some  $w\in [-1+\ep, 1-\ep]$ and $|F_n(z)|\le \exp(n^{c_2})$ for all $z\in B(w, 3/(2n))$. Then
	\begin{equation}
	\int_{B(w, 1/(2n))} \left |\log\left |F_n(z)\right |\right |^{2} dz \le C  n^{-2+6c_2}.\nonumber
	\end{equation}
	
\end{lemma}

\begin{proof}[Proof of Lemma \ref{lm:2norm}]
	We follow ideas from \cite{DOV}; the constant $6$ in the conclusion is adhoc but we make no attempt to optimize it.
		
	From  Jensen's inequality for the number of roots of $P_n^*$, we have
	\[N_{F_n}({B(w, 1/n)}) = N_{P_n^*}({B(w, 1/n)})\le \log \frac{5}{2} (\log M - \log|P^*_n(w)|) < 2   (\log M - \log|P^*_n(w)|)
	\]
	where $N_{F_n}({B(w, 1/n)})$ is the number of zeros of $F_n$ in $B(w, 1/n)$ and $M = \max_{|w - z| = 2/n} |P^*_n(z)|$.
	
	We have
	$$\frac{\max_{|w - z| = 2/n} |P^*_n(z)|}{|P^*_n(w)|} \le \frac{\max_{|w - z| = 2/n} |F_n(z)|}{|F_n(w)|}\sup_{|w - z| = 2/n} \frac{W(a_n w)}{W(a_n z)}\ll \frac{\max_{|w - z| = 2/n} |F_n(z)|}{|F_n(w)|}$$
	where in the last inequality, we used \eqref{eq:Q}.
	
	From this and the assumption of the lemma, we conclude that
	\begin{equation}
	N_{F_n}({B(w, 1/n)})\le 2 n^{c_2}.\label{noncluster}
	\end{equation}
	
	By the pigeonhole principle, there exists a radius $1/n\ge r\ge 1/(2n)$ for which $F_n$ has no zeros in the annulus $B(w, r +\eta)\setminus B(w, r-\eta)$ where $\eta =.1 n ^{-1-c_2}$. We can also assume, without loss of generality, that there is no root on the boundary of each disk.
		
	Let $\zeta_1, \dots, \zeta_m$ be the zeros of $F_n$ in the disk $B(w, r - \eta)$. By \eqref{noncluster},   $m \le 2 n^{c_2}$. Define
	$$p(z) := \frac{P^*_n(z)}{(nz-n\zeta_1)\dots(nz-n\zeta_m)} \quad \text{and}\quad f(z) := p(z) W(a_n z).$$
	
	 Since $p$ is an entire function which does not have zeros in the (closed)  disk $B(w, r +\eta)$, $\log|p|$ is harmonic on this disk. For every $z$ with $|z-w| = r + \eta$, the distance from
	$z$ to any $\zeta_i$ is at least $\eta$, so
	$$|p(z)|\le |P^*_n(z)|n^{-m}\eta^{-m}\ll \exp(n^{c_2})(n\eta)^{-m}W^{-1}(a_n w)=: \exp(n^{c_2})(n\eta)^{-m} \alpha^{-1}$$
	where $\alpha =W(a_n w) $ and we used $\inf_{|z-w|=r+\eta} W(a_n z) = \Theta(W(a_n w))$ by \eqref{eq:Q}.
	
	It follows that for any $z$ where  $|z-w|= r + \eta$
	\begin{equation}\label{upperf} \log|p(z)|\le n^{-c_2} + m \log (n\eta)^{-1 } +\log \alpha+ O(1)\le 21 n^{2c_2}-\log \alpha, \end{equation}   since
	$$n^{c_2} \le n^{2c_2}, m \le 2 n^{c_2}, (n\eta)^{-1} = 10 n^{c_2} \le
	e ^{10 n^{c_2}}. $$
	Because of the harmonicity of $\log|p|$, its maximum
	is achieved on the boundary, and so the same bound holds for all $z\in B(w, r+\eta)$.
	On the other hand, from the lower bound on $|F_n(w)|$ in the lemma
	and the fact that  $|\zeta_i-w|\le 1/n$,
	\begin{equation} \label{lowerf}  \log |p(w)|\ge \log|P^*_n(w)|\ge -n^{c_2}-\log \alpha. \end{equation}
	Now, we make a critical use of Harnack's inequality \cite[Chapter 11]{Ru}, which asserts that if a function  $G$ is harmonic on the open disk $B(w, R)$  and is  nonnegative continuous  on its closure, for some $w\in \C$ and $R>0$, then  for every $z\in B(w, r)$ with $r<R$,
	$$G(z) \le \frac{R+r}{R-r} G(w). $$
	We apply Harnak's inequality  to $G(z):= 21 n^{2c_2} -\log \alpha - \log|p|$ which is
	nonnegative harmonic on $B(w, R)$ with
	$R:= r+ \eta$.  By this inequality, we conclude that for all  $z\in B(w, r)$
	\begin{equation} \label{Harnak1}  21 n^{2c_2}-\log \alpha - \log|p(z)|\le \frac{2r +\eta }{\eta} (21 n^{2c_2}-\log \alpha - \log|p(w)|)\ll n^{2c_2} \frac{2r +\eta }{\eta} . \end{equation}
	As $\eta =.1 n^{-1-c_2}$ and $r < 1$, $\frac{2r+ \eta}{\eta} \ll \eta^{-1} \ll  n^{c_2}$.	It follows that
	$$\log |p(z)|+\log \alpha \gg  -  n^{3c_2}.$$
	
	Together with \eqref{upperf}, we have
	\begin{equation}    |\log |p(z)|+\log \alpha| \ll n^{3c_2}\quad\forall z\in B(w, r). \nonumber\end{equation}
	Thus,
	\begin{equation} \label{absolutef}  |\log |f(z)|| \ll n^{3c_2}\quad\forall z\in B(w, r). \end{equation}
	
	By the triangle inequality and the definition of $f$,
	\begin{equation} \label{Harnak2}
	||{\log |F_n(z)|}||_{L^2(B( w, r))} \le ||{\log |f(z)|}||_{L^2(B( w, r))} + \sum_{i = 1}^{m}||{\log |nz-n\zeta_i|}||_{L^2(B( w, r))}.
	\end{equation}
	
	Notice that  each of the $m$ terms in the sum above is at most $\int_{B(0, 2r-\eta)} |\log |nz||^{2}dz$, as
	$|\zeta_i| \le r -\eta$ for all $i$. As $r < 1/n$, we can further upper bound it by
	$$\int_{B(0, 2/n)} |\log |nz||^{2}dz= n^{-2}\int_{B(0, 2)} |\log |u||^{2}du \ll n^{-2}.$$
	Since $m \le 2 n^{c_2} $, we have
	$$||{\log |F_n(z)|}||_{L^2(B( w, r))} \ll n^{-2+3c_2} $$ which implies  the claim of the lemma as
	$r \ge 1/(2n)$.
	\end{proof}

\subsection{Proof of Theorem \ref{thm:general:real}} \label{pthm:general:real} With our  $\bf {P_n^*}$ plays the role of $F_n$ in \cite{nguyenvurandomfunction17}, the proof is identical to the proof of \cite[Theorem 2.6]{nguyenvurandomfunction17} except that here we have a more relaxed condition (C3) compared to \cite[Condition C2(4)]{nguyenvurandomfunction17} in the second part on the growth of the second derivatives. In particular, in the latter, the condition was
\begin{equation}
\sum_{j=0}^{n} \sup_{z\in B_{\C} (s, 1/n)}  |p_j''(a_n z)|^{2} a_n^4 \le C_1 n ^{4+c_1}\sum_{j=0}^{n} |p_j(a_n s)|^{2}\nonumber
\end{equation}
while here we only require
\begin{equation}
\sup_{z\in B_{\C} (s, 1/n)} \sum_{j=0}^{n} |p_j''(a_n z)|^{2} a_n^4 \le C_1 n ^{4+c_1}\sum_{j=0}^{n} |p_j(a_n s)|^{2}.\nonumber
\end{equation}

The only place in \cite{nguyenvurandomfunction17} where this condition was used is to prove \cite[Equations 24,25]{nguyenvurandomfunction17}. Here, we show how to derive these equations with our new condition. We shall use the notation in \cite{nguyenvurandomfunction17} for this part which says
%
%
$$p(z) = \tilde F(z)-\tilde F(x)-(z-x)\tilde F'(x).$$
So,
$$\E p(z) = \E(\tilde F(z)-\tilde F(x)-(z-x)\tilde F'(z)) = 0.$$
This trivially implies that the left-most side of \cite[Equation 24]{nguyenvurandomfunction17} is smaller than its right-most side.

Next, for \cite[Equation 25]{nguyenvurandomfunction17}, we  observe that
$$p(z) = \int_{x}^{z} \tilde F'(t)dt-(z-x)\tilde F'(x) = \int_{x}^{z}\int_{x}^{t} \tilde F''(w)dwdt.$$
Hence, by Holder's inequality
$$|p(z)|^2 \le \frac{(z-x)^{2}}{2}\int_{x}^{z}\int_{x}^{t} |\tilde F''(w)|^{2}dwdt.$$
$$\Var (p(z) )= \E |p(z)|^2 \le \frac{(z-x)^{4}}{4}\sup_{w\in [x, z]} \E |\tilde F''(w)|^{2}\le \frac{(z-x)^{4}}{4} \sup_{w\in [x, z]}\sum_{j=1}^{n} |\phi_j''(w)|^{2}.$$
Now that we have moved the supremum outside of the summation, we can use our relaxed condition (C3) to get that
  $$\Var (p(z) ) = O\left (\delta_n^{4c_2-c_1} \Var (\tilde F_n (x))\right)$$
  which is \cite[Equation 25]{nguyenvurandomfunction17}.

\section{Almost sure zero distribution results and consequences} \label{sec:zerodistr}

The main goal of this section is to prove Theorem \ref{thm:zeroasympt}.  We start with several auxiliary facts on the asymptotic properties of random variables $\{\xi_n\}_{n=0}^\infty.$

\begin{lemma}\label{lem11.1}
If $\{\xi_n\}_{n=0}^\infty$ are random variables such that $\E[|\xi_n|^{2+\ep_0}] < C$ for some constants $C,\ep_0>0$ and all $n=0,1,\dots$, then
\begin{align} \label{11.1}
\limsup_{n\to\infty} |\xi_n|^{1/n}\le 1 \quad\mbox{ a.s.}
\end{align}
and
\begin{align} \label{11.2}
\limsup_{n\to\infty} \left(\max_{0\le k\le n} |\xi_k| \right)^{1/n} \le 1 \quad\mbox{ a.s.}
\end{align}
\end{lemma}

\begin{proof}[Proof of Lemma \ref{lem11.1}]
It is immediate from our assumptions that
\[
x^{2+\ep_0}\, \P(\{|\xi_n| > x\}) \le \E[|\xi_n|^{2+\ep_0}] < C.
\]
Using the distribution function of $|\xi_n|$ defined by $H_n(x)=\P(\{|\xi_n|\le x\}),\ x\in\R,$ we estimate
\begin{align*}
1 - H_n(x) = \P(\{|\xi_n| > x\}) \le C/x^{2+\ep_0}, \quad x>0,\ n=0,1,\ldots.
\end{align*}
Now \eqref{11.1} follows directly from Lemma 4.1 of \cite{pritsker1}, see (4.1) there.

Let $\CA$ be the event of probability one in \eqref{11.1}. On $\CA$, for any $\eps>0$ there is $n_{\eps}\in\N$ such that $|\xi_n|^{1/n} \le 1 + \eps$
for all $n\ge n_{\eps}$ by \eqref{11.1}. Hence
\[
\max_{0\le k\le n} |\xi_k|^{1/n} \le \max\left(\max_{0\le k\le n_{\eps}} |\xi_k|^{1/n}, 1+\eps \right) \to 1+\eps \quad \mbox{ as } n\rightarrow\infty,
\]
and \eqref{11.2} follows by letting $\eps\to 0.$\end{proof}

In the next two auxiliary lemmas, we also show that under our standard assumptions on the random variables, they cannot be too concentrated.
\begin{lemma}\label{lm:anti}
Let $\{\xi_k\}_{k=0}^\infty$ be independent random variables such that $\E[\xi_k]=0, \Var[\xi_k]=1,$ and $\E[|\xi_k|^{2+\ep_0}] < C$
for some constants $C,\ep_0>0$ and all $k=0,1,\dots$.	There exists a constant $a>0$ depending only on $C$ and $\ep_0$ such that for all $k$,
\begin{equation}\label{anti:a}
\P(|\xi_k|\le a)\le 1-a.
\end{equation}
Moreover, there exists a constant $a'>0$ depending only on $C$ and $\ep_0$ such that for all $k$ and for all $x\in \R$,
\begin{equation}\label{anti:ax}
\P(|\xi_k-x|\le a')\le 1-a'.
\end{equation}
\end{lemma}
\begin{proof}
	To see \eqref{anti:a},  assume that $\P(|\xi_k|\ge a)< a$, then we obtain that
	\begin{eqnarray*}
		1 &=& \E [|\xi_k|^{2}] =\E [|\xi_k|^2 \textbf{1}_{|\xi_k|< a}] + \E [|\xi_k|^{2} \textbf{1}_{|\xi_k|\ge a}] \\
		&\le& a^{2} + \left(\E [|\xi_k|]^{2+\ep_0}\right)^{2/(2+\ep_0)} \left(\P(|\xi|\ge a)\right )^{\ep_0/(2+\ep_0)} \text{ by H\"older's inequality} \\
		&\le& a^{2} + C^{2/(2+\ep_0)} a^{\ep_0/(2+\ep_0)}.
	\end{eqnarray*}
	Since the latter upper bound converges to $0$ as $a\to 0$, there must exist some $a>0$ (independent of $k$) for which it is smaller than 1. This
	contradiction shows that there is a sufficiently small $a>0$ that satisfies \eqref{anti:a}.
	
	Next, for \eqref{anti:ax}, assume that $\P(|\xi_k-x|\le a')> 1-a'$ then letting $\xi_k'$ be an independent copy of $\xi_k$, we have
	 \begin{equation}\label{eq:a'}
	 \P(|\xi_k - \xi_k'|\le 2a')\ge (1-a')^{2}\ge 1-2a'.
	 \end{equation}
	 By applying the first part to the random variables $\frac{\xi_k-\xi_k'}{\sqrt 2}$ which have mean 0, variance 1, and uniformly bounded $(2+\ep_0)$-moments, we obtain an $a'$ that violates \eqref{eq:a'}, and hence satisfies \eqref{anti:ax}.
	\end{proof}

\begin{lemma}\label{lem11.2}
If $\{\xi_k\}_{k=0}^n$ are independent random variables such that $\E[\xi_k]=0, \Var[\xi_k]=1,$ and $\E[|\xi_k|^{2+\ep_0}] < C$
for some constants $C,\ep_0>0$ and all $k=0,1,\dots$, then there is $b>0$ such that
\begin{align} \label{11.3}
\liminf_{n\to\infty} \left(\max_{n-b\log{n}<k\le n} |\xi_k|\right)^{1/n} \ge 1 \quad\mbox{ a.s.}
\end{align}
\end{lemma}

\begin{proof}
Let $s_n\le n,\ n\in\N,$ be a sequence of natural numbers that will be specified later. Consider
\begin{align*}
M_n:=\max_{n-s_n<k\le n} |\xi_k|.
\end{align*}
The statement
\begin{align*}
\liminf_{n\to\infty} \left(M_n\right)^{1/n} \ge 1 \quad\mbox{ a.s.}
\end{align*}
is equivalent to
\begin{align*}
\P(\{M_n \le \lambda^n \mbox{ i.o.}\}) = 0
\end{align*}
for all positive $\lambda<1.$ The latter would follow from the first Borel-Cantelli Lemma, if we show that
\begin{align*}
\sum_{n=1}^\infty \P(\{M_n \le \lambda^n\}) < \infty
\end{align*}
for all positive $\lambda<1.$
Since the variables $\{\xi_k\}_{k=0}^{\infty}$ are independent, we have
\begin{align*}
\P(\{M_n \le \lambda^n\}) = \prod_{k=1}^{s_n }\P(\{|\xi_{n-k}| \le \lambda^n\}).
\end{align*}
Using \eqref{anti:a}, for any $\lambda<1$ we find $N=N(a,\lambda)\in\N$ such that $\P(\{|\xi_k| \le \lambda^n\})\le 1-a$ for all $n\ge N$ and all $k=0,1,\dots.$ This gives
\begin{align*}
\sum_{n=1}^\infty \P(\{M_n \le \lambda^n\}) \le \sum_{n=1}^\infty (1-a)^{s_n} < \infty,
\end{align*}
provided $(1-a)^{s_n} \le 1/n^2$ for large $n$. It suffices to take $s_n \ge (-2/\log(1-a))\log{n}$ to satisfy the latter condition.
\end{proof}

We now state a companion result for the coefficients of our contracted random orthogonal polynomials
\[
P_n^*(x)=P_n(a_n x)=\sum_{k=0}^n \xi_k p_k(a_n x)=\sum_{k=0}^n c_{k,n} x^k,
\]
see \eqref{1.5}.

\begin{lemma}\label{lem11.3}
Let $\{\xi_k\}_{k=0}^n$ be independent random variables such that $\E[\xi_k]=0, \Var[\xi_k]=1,$ and $\E[|\xi_k|^{2+\ep_0}] < C$
for some constants $C,\ep_0>0$ and all $k=0,1,\dots$. Assume that $W=e^{-Q}\in \mathcal{F} (C^2)$, where $Q$ is even, and the function
$T$ in the definition of $\mathcal{F} (C^2)$ satisfies \eqref{1.3} with $\alpha\in(1,\infty)$. For the orthogonal polynomials $p_k$ that
correspond to the weight $W$, write
\[
p_k(a_n x) = \sum_{j=0}^k b_{j,k} x^j = b_{k,k} \prod_{j=1}^k (x-x_{j,k}),\quad k=0,1,\ldots,n.
\]
If the random variables $c_{k,n}$ are defined by
\[
c_{k,n}:= \sum_{i=k}^n \xi_i b_{k,i},\quad  k=0,1,\ldots,n,
\]
then there is $b>0$ such that
\begin{align} \label{11.5}
\liminf_{n\to\infty} \left(\max_{n-b\log{n}<k\le n} |c_{k,n}|\right)^{1/n} \ge 2\, e^{1/\alpha} \quad\mbox{ a.s.}
\end{align}
\end{lemma}	
	
\begin{proof} We follow the idea in \cite[Lemma 4.2]{dauvergne2021necessary}.
Note that $b_{k,k}=a_n^k \gamma_k$, where $\gamma_k$ is the leading coefficient of $p_k$. We have
\begin{equation}\label{eq:limbnn}
\lim_{n\to\infty} |b_{n,n}|^{1/n} = \lim_{n\to\infty} a_n \gamma_n^{1/n} = 2\, e^{1/\alpha}.
\end{equation}
For every $\ep>0$, it suffices to show that
\begin{equation}\label{eq:BC1}
\P\left (\left (\max_{n-b\log{n}<k\le n} |c_{k,n}|\right)^{1/n} \le 2 e^{1/\alpha} -\ep \text{ infinitely often}\right )=0.
\end{equation}
Note that $|c_{k,n}|\le (2 e^{1/\alpha} -\ep)^{n}$ implies that $\xi_k b_{k, k}$ lies in an interval of radius $(2 e^{1/\alpha} -\ep)^{n}$ centered around $\sum_{i=k+1}^n \xi_i b_{k,i}$. Thus, by independence, for any event $A_{k, n}\in \sigma(\xi_{k+1}, \dots, \xi_n)$ -- the $\sigma$-algebra generated by $\xi_{k+1}, \dots \xi_n$, we have
\begin{eqnarray}
\P\left (|c_{k,n}|^{1/n} \le 2\, e^{1/\alpha} -\ep\big\vert A_{k, n}\right ) &\le&\sup_{x\in \R} \P\left ( |\xi_k b_{k, k}-x| \le \left (2\, e^{1/\alpha} -\ep\right )^{n}\right )\notag.\end{eqnarray}
By \eqref{eq:limbnn}, this implies that for sufficiently large $n$ and $k\ge n-b\log n$,
\begin{eqnarray}
\P\left (|c_{k,n}|^{1/n} \le 2\, e^{1/\alpha} -\ep\big\vert A_{k, n}\right ) &\le&\sup_{x\in \R} \P\left ( |\xi_k-x| \le \left (1 -\delta\right )^{n}\right ) \notag
\end{eqnarray}
where $\delta = \frac{\ep}{2\, e^{1/\alpha}}$ is a constant.

By \eqref{anti:ax}, the last quantity is smaller than $1 - a'$ for the same constant $a'$ as in \eqref{anti:ax}, again for all sufficiently large $n$. Thus, by setting $$A_{k, n}=\{|c_{\ell,n}|^{1/n}\le   2\, e^{1/\alpha} -\ep, \forall \ell\in (k , n]\cap \N\}\in \sigma(\xi_{k+1}, \dots, \xi_n)$$ and $A_{n, n}$ be the whole sample space, we get
\begin{eqnarray}
 \P\left (\left (\max_{n-b\log{n}<k\le n} |c_{k,n}|\right)^{1/n} \le 2\, e^{1/\alpha} -\ep\right ) 
 &=& \prod_{k=\lfloor n-b\log n\rfloor+1}^{n}\P\left (|c_{k,n}| \le (2\, e^{1/\alpha} -\ep)^{n}\big\vert A_{k, n}\right ) \notag\\
 &\le& (1-a')^{b\log n}\le n^{-a'b}\notag.
\end{eqnarray}
By choosing $b = 2/a'$, the right-most side becomes $n^{-2}$ which is summable over $n$.
By the first Borel-Cantelli Lemma, this implies \eqref{eq:BC1}.
  \end{proof}

\begin{proof}[Proof of Theorem \ref{thm:zeroasympt}]
This proof is similar to the proof of Theorem 2.4 in \cite{LPX2}, but it has several additional steps due to different assumptions on random
coefficients. Recall that the standard Freud weight with index $\alpha$ is given by
\[
w(s)=e^{-\gamma_\alpha |s|^\alpha}, \quad s\in \R,
\]
where
\[
\gamma_\alpha = \frac{\Gamma(\frac{\alpha}{2})\Gamma(\frac{1}{2})}{2\Gamma(\frac{\alpha}{2}+\frac{1}{2})} = \int_0^1 \frac{t^{\alpha-1}}{\sqrt{1-t^2}}\, dt
\]
see \cite[p. 239]{ST}. Note that by \cite[p. 240]{ST}, $F_w=\log 2+1/\alpha$ is the modified Robin constant and $\mu_w=\mu_\alpha$ is the equilibrium measure corresponding to $w$. Following \cite{ST}, we call a sequence of monic polynomials $\{M_n\}_{n=1}^\infty$, with $ \deg(M_n)=n,$ asymptotically extremal with respect to the weight $w$ if it satisfies
\begin{equation}\label{11.11}
\lim_{n\to\infty} \|w^n M_n\|_{L^\infty(\R)}^{1/n} = e^{-F_w} = e^{-1/\alpha}/2.
\end{equation}
Theorem 4.2 of \cite[p. 170]{ST} states that  the normalized zero counting measures of $M_n$ converge weak* to $\mu_w=\mu_\alpha$. On the other hand, by Corollary 2.6 of \cite[p. 157]{ST} and Theorem 5.1 of \cite[p. 240]{ST},
\[
\|w^n M_n\|_{L^\infty(\R)}=\|w^n M_n\|_{L^\infty([-1,1])}.
\]
Applying Theorem 3.6 of \cite[p. 46]{ST}, we see that \eqref{11.11} is equivalent to
\begin{equation}\label{11.12}
\limsup_{n\to\infty} \|w^n M_n\|_{L^\infty([-1, 1])}^{1/n} \leq e^{-F_w} = e^{-1/\alpha}/2.
\end{equation}
We construct the monic polynomials $M_n$ from $P_n^*$ in such a way that they share most of zeros, and therefore have the same limit of their zero counting measures.
Continuing with our previous notation, we write
\[
P_n^*(z)=P_n(a_n z)=\sum_{k=0}^n \xi_k p_k(a_n z)=\sum_{k=0}^n c_{k,n} z^k = c_{d_n,n} \prod_{k=1}^{d_n} (z-z_{k,n}),\ c_{d_n,n}\neq 0,
\]
where we assume that the zeros $\{z_{k,n}\}_{k=1}^{d_n}$ are listed in the order of decreasing absolute values: $|z_{1,n}|\ge |z_{2,n}|\ge\ldots\ge |z_{n,n}|,$ and where
$d_n$ is the actual degree of $P_n^*.$ Note that $n-d_n = o(n)$ with probability one by \eqref{11.5}. Let $m_n$ be the number of zeros of $P_n^*$ with absolute value greater than 2,
and consider $l_n=\min(m_n,b\log{n})$, where $b>0$ is a constant from Lemma  \ref{lem11.3}. We define the monic polynomial $M_n$ of degree $n$ by
\[
M_n(z) := \frac{(z-2)^{n-d_n+l_n}\,P_n^*(z)}{c_{d_n,n}\, \prod_{1\le k\le l_n} (z-z_{k,n})},\quad n\in\N,
\]
and show they are asymptotically extremal in the sense of \eqref{11.11}-\eqref{11.12} with probability one. Thus the zero counting measures of $M_n$ converge weak* to $\mu_\alpha$ as discussed above, and the same conclusion is true for the zero counting measures of $P_n^*$, because the sets of roots for $M_n$ and $P_n^*$ differ only by $o(n)$ elements.

Using Vieta's formulas for $P_n^*$, we obtain that
\[
|c_{n-k,n}| \le |c_{d_n,n}| \left| \sum_{j_1<\ldots<j_k} z_{j_1,n}\ldots z_{j_k,n} \right| \le |c_{d_n,n}| {n\choose k} \prod_{j=1}^k |z_{j,n}| \le n^k |c_{d_n,n}| \prod_{j=1}^k |z_{j,n}|,\ k\le l_n,
\]
where we used decreasing ordering of zeros. Since $l_n=O(\log{n})$, it follows from the last inequality and Lemma \ref{lem11.3} that
\[
\liminf_{n\to\infty} \left( |c_{d_n,n}| \prod_{j=1}^k |z_{j,n}| \right)^{1/n} \ge  \liminf_{n\to\infty} \left(\max_{1\le k\le l_n} |c_{n-k,n}|\right)^{1/n} \ge 2\, e^{1/\alpha} \quad\mbox{ a.s.}
\]
Further, we apply the above inequality and that $n-d_n+l_n=o(n)$ to estimate (with probability one)
\begin{align*}
\limsup_{n\to\infty} \left\|M_n w^n \right\|_{L^\infty([-1, 1])}^{1/n} &\le \limsup_{n\to\infty} \left\|P_n^* w^n \right\|_{L^\infty([-1, 1])}^{1/n}
\left\| \frac{(z-2)^{n-d_n+l_n}}{c_{d_n,n}\, \prod_{1\le k\le l_n} (z-z_{k,n})}\right\|_{L^\infty([-1, 1])}^{1/n} \\
&\le \limsup_{n\to\infty} \left\|P_n^* w^n \right\|_{L^\infty([-1, 1])}^{1/n}
\left( \frac{3^{n-d_n+l_n}}{|c_{d_n,n}|\, \prod_{1\le k\le l_n} (|z_{k,n}|/2)}\right)^{1/n} \\
&\le \frac{1}{2e^{1/\alpha}}\limsup_{n\to\infty} \left\|P_n^* w^n \right\|_{L^\infty([-1, 1])}^{1/n} \\ &=e^{-F_w}\limsup_{n\to\infty} \left\|P_n^* w^n \right\|_{L^\infty([-1, 1])}^{1/n}.
\end{align*}
For the second inequality in the above display, we also used that $|z_{k,n}|\ge 2,\ 1\le k\le l_n,$ by our construction, so that $|z-z_{k,n}|\ge |z_{k,n}|-1\ge |z_{k,n}|/2 $
for $z\in[-1,1],\ 1\le k\le l_n,$ and clearly $|z-2|\le 3$ for $z\in[-1,1].$

Thus the proof would be completed if we show that
\begin{equation}\label{11.13}
\limsup_{n\to\infty} \left\|P_n^* w^n \right\|_{L^\infty([-1, 1])}^{1/n} \le 1\quad \mbox{a.s.}
\end{equation}
Recall that by \eqref{3.0}
\begin{align*}
\left\|W P_n\right\|_{L^{\infty}(\R)} \leq C n \left(\sum_{k=0}^n |\xi_k|^2\right)^{1/2},\quad n\in\N.
\end{align*}
Changing variable $x=a_n s$ and introducing $w_n(s)=W(a_n s)^{1/n},$ we obtain that
\[
\left\|P_n^* w_n^n \right\|_{L^\infty([-1, 1])} = \left\|P_nW \right\|_{L^\infty([-a_n, a_n])} = \left\|P_nW \right\|_{L^\infty(\R)}
\]
by \cite[p. 4]{Lubinskybookexp}. Hence
\[
\limsup_{n\to\infty} \left\|P_n^* w_n^n \right\|_{L^\infty([-1, 1])}^{1/n} \leq \limsup_{n\to\infty} \left(\max_{0\le k\le n} |\xi_k| \right)^{1/n} \le 1 \quad\mbox{ a.s.}
\]
by \eqref{11.2} of Lemma \ref{lem11.1}. It follows that, with probability one, we have
\begin{align*}
\limsup_{n\to\infty} \left\|P_n^* w^n \right\|_{L^\infty([-1, 1])}^{1/n} &\le \limsup_{n\to\infty} \left\|P_n^* w_n^n \right\|_{L^\infty([-1, 1])}^{1/n} \left\|w/w_n\right\|_{L^\infty([-1, 1])} \\ &\le \limsup_{n\to\infty} \left\|w/w_n\right\|_{L^\infty([-1,1])}.
\end{align*}
Since $w_n$ and $w$ are both even, to prove \eqref{11.13} it only remains to show that
\[
\limsup_{n\to\infty} \left\|w/w_n\right\|_{L^\infty([0,1])} \le 1.
\]
But exactly this limit relation was established in the end of proof for Theorem 2.4 in \cite{LPX2}, so that we refer there to avoid duplication.
\end{proof}

\begin{proof}[Proof of Corollary \ref{cor:explimit}]
The proof of this result is identical to the proof of Corollary 2.5 in \cite{LPX2}. Thus we omit it.
\end{proof}

 \section{Proof of Theorem \ref{thm:edge}}\label{sec:proof:edge}

 We recall that $D(\ep)= [-1+\ep, 1-\ep]$. We need to show that the number of real roots outside of $D(\ep)$ is negligible, i.e., 
 \begin{equation*}\label{key}
 \lim_{\ep\to 0} \lim _{n\to \infty}\frac{\E[N_{\ps}(\R \setminus D(\ep))]}{n} = 0.
 \end{equation*}
The corresponding result for $\E[N_{\wps}]$ is included in the above statement. 

For any $\ep>0,$ let $F_\ep:=([-1/\ep,\ep-1]\cup[1-\ep,1/\ep])\times [-\ep, \ep]$. Recall from \eqref{Ullman} that $\mu_\alpha$ is absolutely continuous, and supp$\,\mu_\alpha = [-1,1].$ It follows that $\mu_{\alpha}(\partial F_\ep)=0$, and we can apply Corollary \ref{cor:explimit} to obtain that
\[ 
\lim_{n\rightarrow\infty}\frac{\E N_{\ps}(F_\ep)}{n}=\mu_{\alpha}(F_\ep).
\]
Furthermore, monotone convergence and basic measure properties give that
\begin{align*}
\lim_{\ep\to 0} \lim _{n\to \infty}\frac{\E[N_{\ps}(\R \setminus D(\ep))]}{n} &\le \limsup_{\ep\to 0} \lim _{n\to \infty}\frac{\E[N_{\ps}(F_\ep)]}{n} = \limsup_{\ep\to 0} \mu_{\alpha}(F_\ep) \\
&= \limsup_{\ep\to 0} \mu_{\alpha}([-1,\ep-1]\cup[1-\ep,1]) = 0,
\end{align*}
which completes this proof.

\section{Correlations of real roots}\label{sec:correlation}
For each $1\le k\le n$, let $\rho_k(x_1,\dots, x_k)$ be the $k$-correlation of the real roots of $P_n(x)$ (see for instance \cite{HKPV}), for which 
$$\E \bigg[\sum G\left (\zeta_{i_1}, \dots, \zeta_{i_k}\right)\bigg]  = \int_{\R^k} G(x_1,\dots, x_k) \rho_k(x_1,\dots, x_k) dx_1 \dots dx_k,$$
for any continuous, compactly supported test function $G : \R^k \to \R$, where the sum runs over all $k$-tuples $(\zeta_{i_1}, \dots, \zeta_{i_k})$ of the real roots of $P_n(x)$. 

In principle one can use Kac-Rice formula (see also \cite{HKPV}\footnote{The literature on this formula is so vast that it is impossible to list even a small portion of it.}) to compute these correlation functions, that
\begin{equation}\label{eqn:KR}
\rho_k(x_1,\dots, x_k) = \int_{\R^k} |y_1,\dots, y_k|p(\mathbf{0},\By) dy_1\dots dy_k,
\end{equation}
where $p(.)$ is the joint density function of the random vectors $(P_n(x_1),\dots, P_n(x_k))$ and $(P_n'(x_1),\dots, P_n'(x_k))$. 

This formula is especially useful when $\xi_i$ are iid standard gaussian and when $k$ is small. For instance, directly related to our current setting of orthogonal polynomials, the cases $k=1$ and $k=2$ were computed in \cite{LPX2} and \cite{LP} respectively. However, the formulas become increasingly more complicated when $k$ gets larger and for other ensembles of random coefficients. Our goal in this section is to provide some alternative formulas by following \cite{GKZ}.

For each $\Bx=(x_1,\dots, x_k)\in \R^k$,  denote by $V(\Bx)$ the Vandermonde-type matrix
\[
V(\Bx)=
\begin{pmatrix}
p_0(x_1) & p_1(x_1) & \dots & p_{k-1}(x_1) \\
\vdots & \vdots & \ddots & \vdots \\
p_0(x_k) & p_1(x_k) & \dots & p_{k-1}(x_k) \\
\end{pmatrix},
\]
where we recall that $p_i(x)$ are the orthogonal polynomials of degree $i$ with respect to $\mu$. Assume that $p_i(x) =\gamma_i x^i+\dots$, then from the recurrence formula \eqref{eqn:rec} we have $\gamma_i = (A_0 \dots A_{i-1})^{-1} \gamma_0$ and 
\begin{equation*}\label{eqn:det}
\det(V(\Bx))=\prod_{m=0}^{k-1} \gamma_m \prod_{1\leq i<j\leq k}(x_j-x_i).
\end{equation*}
As the random coefficients $\xi_i$ of $P_n(x)$ have density, with probability one the roots of $P_n(x)$ are distinct, and hence we will assume that the $x_i$ are distinct. Consider the random function $\Bet=(\eta_0,\dots,\eta_{k-1})^T:\R^k\to\R^k$ defined as
\begin{equation}\label{eqn:eta}
  \Bet(\Bx):=
  -V^{-1}(\Bx)
\begin{pmatrix}
\sum_{j=k}^n\xi_j p_j(x_1)\\
\vdots\\
\sum_{j=k}^n\xi_j p_j(x_k)
\end{pmatrix}.
\end{equation}

Our main result of this section is the following formula.

\begin{theorem}\label{thm:k:1} Let $f_0,\dots, f_{k-1}$ be the density functions of $\xi_0,\dots, \xi_{k-1}$ respectively. With $\Bet$ from \eqref{eqn:eta} we have
\begin{align*}
\rho_k(\Bx) &= \prod_{m=0}^{k-1} \gamma_m^{-1}  \prod_{1\leq i<j\leq k}|x_i-x_j|^{-1} \E\left(\prod_{i=1}^{k}\left |\sum_{j=0}^{k-1}\eta_j(\Bx)p_j(x_i)+\sum_{j=k}^{n}\xi_j p_j(x_i)\right |\prod_{i=0}^{k-1}f_i(\eta_i(\Bx))\right), \nonumber
\end{align*}
where the expectation is with respect to ${\xi_k,\dots, \xi_n}$.
\end{theorem}
Alternatively, we also have the following analog of \cite[Theorem 2.3]{GKZ}, where $\sigma_i(\Bx), 1\le i\le k$, are the elementary symmetric polynomials 
$\sum_{1\le j_1 < \dots < j_i\le k} x_{j_1} x_{j_2} \dots x_{j_i}$ (with the convention that $\sigma_0(\Bx)=1$) and where for short we write
$$\langle p(x),q(x) \rangle_\mu := \int_\R p(x)q(x) d\mu(x).$$

\begin{theorem}\label{thm:k:2} Let $f_0,\dots, f_{n}$ be the density functions of $\xi_0,\dots, \xi_{n}$ respectively. We have
\begin{align*}
\rho_k(\mathbf{x}) &=  \prod_{m=0}^{n} \gamma_m^{-1}  \prod_{1\le i < j \le k} |x_i - x_j| \times \\
&\times \int\limits_{\mathbb{R}^{n-k+1}} \prod_{l=0}^{n}f_l\bigg(\sum_{i=l}^n  \big(\sum_{j=0}^{n-k}(-1)^{k-i+j}\sigma_{k-i+j}(\Bx)t_j\big) \langle x^i, p_l(x) \rangle_\mu \bigg) \prod_{i=1}^k |\sum_{j=0}^{n-k} t_j x_i^j| \, dt_0\dots dt_{n-k},
\end{align*}
\end{theorem}

These formulas seem to be useful when $k$ is comparable to $n$. For instance when $k=n$ we obtain the following joint density formula for $n$ real roots of $P_n(x)$. 
\begin{corollary}
\begin{align*}
\rho_n(\mathbf{x}) &=  \prod_{m=0}^{n} \gamma_m^{-1}  \prod_{1\le i < j \le n} |x_i - x_j| \int\limits_{\mathbb{R}} \prod_{l=0}^{n}f_l\bigg(\big(\sum_{i=l}^n (-1)^{n-i}\sigma_{n-i}(\Bx) \langle x^i, p_l(x) \rangle_\mu\big) t \bigg) t^n \, dt.
\end{align*}
\end{corollary}

In what follows we discuss the proof of the above results. Notice that $P_n(x_1)=\dots = P_n(x_k)=0$ if and only if

\begin{equation*}
\begin{pmatrix}
p_0(x_1) & p_1(x_1) & \dots & p_n(x_1) \\
\vdots & \vdots & \ddots & \vdots \\
p_0(x_k) & p_1(x_k) & \dots & p_n(x_k) \\
\end{pmatrix}
\begin{pmatrix}
\xi_0\\
\vdots\\
\xi_n
\end{pmatrix}
=\mathbf{0},
\end{equation*}
which can be rewritten as
$$
\begin{pmatrix}
p_k(x_1) & p_{k+1}(x_1) & \dots & p_n(x_1) \\
\vdots & \vdots & \ddots & \vdots \\
p_k(x_k) & p_{k+1}(x_k) & \dots & p_n(x_k) \\
\end{pmatrix}
\begin{pmatrix}
\xi_k\\
\vdots\\
\xi_n
\end{pmatrix}
=-V(\Bx)
\begin{pmatrix}
\xi_0\\
\vdots\\
\xi_{k-1}
\end{pmatrix}.
$$
Thus
$$
\Bet(\Bx)=
 \begin{pmatrix}
\xi_0\\
\vdots\\
\xi_{k-1}
\end{pmatrix}.
$$

Denote by $J_{\Bet}(\Bx)$ the Jacobian matrix of $\Bet$ at the point $\Bx$.
\begin{lemma}We have
$$
\det (J_{\Bet}(\Bx)) = (-1)^k \frac{\prod_{i=1}^{k}\left(\sum_{j=0}^{k-1}\eta_j(\Bx)p_j'(x_i) + \sum_{j=k}^{n}\xi_j p_j'(x_i)\right)}{\prod_{m=0}^{k-1} \gamma_m \prod_{1\leq i<j\leq k}(x_j-x_i)}.
$$
\end{lemma}

\begin{proof} Recall that

\[
V(\Bx) \Bet(\Bx)=
  -
\begin{pmatrix}
p_k(x_1) & p_{k+1}(x_1) & \dots & p_n(x_1) \\
\vdots & \vdots & \ddots & \vdots \\
p_k(x_k) & p_{k+1}(x_k) & \dots & p_n(x_k) \\
\end{pmatrix}
\begin{pmatrix}
\xi_k\\
\vdots\\
\xi_n
\end{pmatrix}.
\]
By differentiating,
\begin{eqnarray*}
  V(\Bx)J_{\Bet}(\Bx)&+&\diag\left(\sum_{j=0}^{k-1}\eta_j(\Bx)p_j'(x_1),\dots,\sum_{j=0}^{k-1}\eta_j(\Bx)p_j'(x_k)\right)\nonumber\\
   &=&-\diag\left(\sum_{j=k}^{n}\xi_j p_j'(x_1),\dots,\sum_{j=k}^{n}\xi_j p_j'(x_k)\right).
\end{eqnarray*}
It thus follows that 
$$
\det (J_{\Bet}(\Bx)) =
(-1)^k  \det(V(\Bx))^{-1} \prod_{i=1}^{k}\left(\sum_{j=0}^{k-1}\eta_j(\Bx)p_j'(x_i) + \sum_{j=k}^{n}\xi_j p_j'(x_i)\right)
$$
giving the desired result.
\end{proof}
We next rely on the Coarea formula \cite{Federer}.

\begin{lemma}
Let $B\subset \R^k$ be a region. Let $u:B\to\R^k$ be a Lipschitz function and $h:\R^k\to\R^1$ be an $L^1$-function. Then
\[
\int\limits_{\R^k}\#\{\Bx\in B\,:\, u(\Bx)=\By\}\,h(\By)\,d\By=\int\limits_{B}|\det J_u(\Bx)|\, h(u(\Bx))\,d\Bx,
\]
where $J_u(\Bx)$ is the Jacobian matrix of $u(\Bx)$.
\end{lemma}

\begin{proof} [Proof of Theorem \ref{thm:k:1}] Let $B_1,\dots,B_k$ be a family of mutually disjoint Borel subsets  in $\R$ and let $B=B_1\times\dots\times B_k$. By the Coarea formula and by Fubini theorem we have
\begin{align*}
\E\,\left[\prod_{i=1}^{k}N_n(B_i)\right]&=\E\,\#\Big\{\Bx\in B\,:\, \eta(\Bx)=(\xi_0,\dots,\xi_{k-1})\Big\}\\
&=\E\,\int\limits_{\R^k}\,\#\Big\{\Bx\in B\,:\, \eta(\Bx)=\By\Big\}f_0(y_0)\dots f_{k-1}(y_{k-1})\,d\By \\
& =\int\limits_{B}\E\,|\det (J_{\Bet}(\Bx))|\, f_0(\eta_0(\Bx))\dots f_{k-1}(\eta_{k-1}(\Bx))\,d\Bx,
\end{align*}
where we recall that $f_0,\dots, f_{k-1}$ are the density functions of $\xi_0,\dots, \xi_{k-1}$ respectively.
\end{proof}

\begin{proof} [Proof of Theorem \ref{thm:k:2}] Note that by Theorem \ref{thm:k:1},
\begin{align}\label{eqn:rho:start}
\rho_k(\Bx) &=   \prod_{m=0}^{k-1} \gamma_m^{-1} \prod_{1\leq i<j\leq k}|x_j-x_i|^{-1}\\
&\times\int\limits_{\R^{n-k+1}}\prod_{i=1}^{k}\left|\sum_{j=0}^{n}a_j p_j'(x_i)\right|\prod_{i=0}^{n}f_i(a_i)\, da_k da_{k+1}\dots da_n, \nonumber
\end{align}
where $(a_0, \dots, a_{k-1})^T =   -V^{-1}(\Bx) (\sum_{j=k}^n a_j p_j(x_1),\dots, \sum_{j=k}^n a_j p_j(x_k))^T$.

This means that $x_1,x_2,\dots,x_k$ are zeros of the polynomial $\sum_{i=0}^n a_i p_i(x)$. Hence there exists a unique polynomial $\sum_{j=0}^{n-k} b_j x^j$ such that
\begin{align*}
P(x)=\sum_{i=0}^n a_i p_i(x) =\prod_{i=1}^k(x-x_i)\bigg(\sum_{j=0}^{n-k} b_j x^j \bigg)&= \bigg(\sum_{j=0}^k (-1)^{k-j} \sigma_{k-j}(\Bx) x^j\bigg) \bigg(\sum_{j=0}^{n-k} b_j x^j \bigg)\\ 
&= \sum_{i=0}^n \big(\sum_{j=0}^{n-k}(-1)^{k-i+j}\sigma_{k-i+j}(\Bx)b_j\big)x^i.
\end{align*}
The variables $a_0,\dots,a_n$ are uniquely defined by $\Bx$ and $b_0,\dots,b_{n-k}$ from the above equation, that for any $0\le l\le n$ we have 
$$a_l= \langle P(x), p_l(x) \rangle_\mu =  \sum_{i=l}^n  (\sum_{j=0}^{n-k}(-1)^{k-i+j}\sigma_{k-i+j}(\Bx)b_j) \langle x^i, p_l(x) \rangle_\mu$$
where we note that $\langle x^i, p_l(x) \rangle_\mu =0$ if $i\le l-1$. 
In particularly,
\begin{align*}
a_k& = \sum_{i=k}^n  (\sum_{j=0}^{n-k}(-1)^{k-i+j}\sigma_{k-i+j}(\Bx)b_j) \langle x^i, p_k(x) \rangle_\mu \\
&= (\sum_{j=0}^{n-k}(-1)^{j}\sigma_{j}(\Bx)b_j) \langle x^k, p_k(x) \rangle_\mu +  (\sum_{j=1}^{n-k}(-1)^{j-1}\sigma_{j-1}(\Bx)b_j) \langle x^{k+1}, p_k(x) \rangle_\mu  +\dots\\
& +(\sum_{j=n-k}^{n-k}(-1)^{k-n+j}\sigma_{k-n+j}(\Bx)b_j) \langle x^n, p_k(x) \rangle_\mu\\
&= b_0  \langle x^k, p_k(x) \rangle_\mu + b_1 c_{k,1}+ \dots + b_{n-k} c_{k, n-k},
\end{align*}
for some numbers $c_{k,1},\dots, c_{k,n-k}$ independently of the $b_i$.

More generally, for any $0\le l\le n-k$
\begin{align*}
a_{k+l}& = \sum_{i=k+l}^n  (\sum_{j=0}^{n-k}(-1)^{k-i+j}\sigma_{k-i+j}(\Bx)b_j) \langle x^i, p_{k+l}(x) \rangle \\
&= (\sum_{j=l}^{n-k}(-1)^{j-l}\sigma_{j-l}(\Bx)b_j) \langle x^{k+l}, p_{k+l}(x) \rangle \\
&+  (\sum_{j=l+1}^{n-k}(-1)^{j-l-1}\sigma_{j-l-1}(\Bx)b_j) \langle x^{k+l+1}, p_{k+l}(x) \rangle  +\dots+\\
& +(\sum_{j=n-k}^{n-k}(-1)^{k-n+j}\sigma_{k-n+j}(\Bx)b_j) \langle x^n, p_{k+l}(x) \rangle\\
&= b_l  \langle x^{k+l}, p_{k+l}(x) \rangle + b_{l+1} c_{k+l,1}+ \dots + b_{n-k} c_{k+l, n-k-l}
\end{align*}
for some numbers $c_{k+l,1},\dots, c_{k,n-k-l}$ independently of the $b_i$.

Thus the Jacobian of the substitution of $(a_k,\dots, a_n)$ by $(b_0,\dots, b_{n-k})$ is a lower triangle matrix with diagonal $\diag(\langle x^k, p_k(x) \rangle_\mu, \dots, \langle x^{n}, p_{n}(x) \rangle_\mu) =\diag(\gamma_k^{-1}, \dots, \gamma_n^{-1})$. Hence 
\begin{equation}\label{eqn:ab}
da_k \dots d a_n =\prod_{j=k}^n \gamma_j^{-1}  d b_0 \dots d b_{n-k}.
\end{equation}
Finally, by differentiating $P(x) = \sum_{j=0}^n a_j p_j(x) = \prod_{j=1}^k(x-x_j)\left(\sum_{j=0}^{n-k} b_j x^j \right)$ at the point $x_i$ we get
\begin{equation*}
\sum_{j=0}^{n}a_j p'(x_i)=\prod_{j \in \{1,\dots, k\} \backslash\{i\}} (x_j - x_i) (\sum_{j=0}^{n-k} b_j x_i^j) ,\quad i=1,\dots k.
\end{equation*}
Substitute the above and \eqref{eqn:ab} into \eqref{eqn:rho:start} we obtain the claim.
\end{proof}

\medskip
\noindent
\textbf{Acknowledgements.} \newline
This paper was written within the AIM SQuaREs project ``Random Polynomials." All authors would like to gratefully acknowledge support of the American Institute of Mathematics. Y. Do was supported by DMS-1800855. D. Lubinsky was supported by NSF grant DMS-1800251. H. H. Nguyen was supported by NSF CAREER grant DMS-1752345. O. Nguyen was supported by NSF grant DMS-2125031.
I. Pritsker was partially supported by NSA grant H98230-21-1-0008, NSF grant DMS-2152935, and by the Vaughn Foundation endowed Professorship in Number Theory.

\bibliographystyle{plain}
\bibliography{polyref}
\end{document}